\documentclass[12pt,reqno]{amsart}
\usepackage{fullpage,amsfonts,amsmath,amssymb}

\usepackage{amsfonts,amsmath,amssymb,fullpage,graphicx}
\usepackage{verbatim}
\usepackage{hyperref}
\usepackage{ragged2e}
\usepackage{xcolor}

\newcommand{\beas}{\begin{eqnarray*}}
\newcommand{\eeas}{\end{eqnarray*}}
\newcommand{\bes} {\begin{equation*}}
\newcommand{\ees} {\end{equation*}}
\newcommand{\be} {\begin{equation}}
\newcommand{\ee} {\end{equation}}
\newcommand{\bea} {\begin{eqnarray}}
\newcommand{\eea} {\end{eqnarray}}

\newcommand{\R}{\mathbb R}
\newcommand{\C}{\mathbb C}
\newcommand{\Z}{\mathbb Z}%
\newcommand{\N}{\mathbb N}

\newcommand{\X}{\mathbb{X}}
\newcommand{\what}{\widehat}
\newcommand{\wtilde}{\widetilde}

\newcommand{\ol}{\overline}
\newcommand{\mf}{\mathfrak}

\newtheorem{theorem}{Theorem}[section]
\newtheorem{lemma}[theorem]{Lemma}
\newtheorem{proposition}[theorem]{Proposition}
\newtheorem{corollary}[theorem]{Corollary}

\theoremstyle{definition}

\theoremstyle{definition}
\newtheorem{remark}[theorem]{Remark}

\numberwithin{equation}{section}

\begin{document}
		\title{$L^p$-boundedness of Pseudo differential operators on rank one Riemannian symmetric spaces of noncompact type}

\author{Sanjoy Pusti, Tapendu Rana}	

\address{Department of Mathematics, IIT Bombay, Powai, Mumbai-400076, India}
\email{sanjoy@math.iitb.ac.in, tapendurana@gmail.com}

\begin{abstract}
The aim of this paper is to study $L^p$-boundedness property of the pseudo differential operator associated with a symbol, on rank one Riemannian symmetric spaces of noncompact type, where the symbol satisfies H\"ormander-type conditions near infinity.  
\end{abstract}


\thanks{Second author  is supported by a Ph.D fellowship of the CSIR(India).}
\maketitle

\section{\textbf{Introduction}}

A differential operator $p(x, D)=\sum_{|\alpha|\leq m} a_{\alpha}(x) \frac{\partial^\alpha}{\partial x^\alpha}$ can be represented as 
\bes
p(x, D) f(x)=\int_{\R^n} e^{2\pi i x\cdot\xi} p(x, \xi)\mathcal F f(\xi)\,d\xi,
\ees
where $p(x, \xi)=\sum_{|\alpha|\leq m} a_\alpha(x) \xi^\alpha$ and $\mathcal F f$ is the  Fourier transform of $f$  defined by \begin{equation*}
\mathcal F f(\xi) = \int_{\R^n} f(x)e^{-2\pi i x\cdot \xi}  dx.
		 \end{equation*}
Let $a(x, \xi)$ be a general function on $\R^n \times \R^n$ not necessarily a polynomial in the $\xi$ variable. Consider the operator $a(x, D)$ defined by
\bes
a(x, D)f(x)=\int_{\R^n} e^{2\pi i x\cdot\xi} a(x, \xi)\mathcal F f(\xi)\,d\xi.
\ees	
This operator $ a(x,D) $ is called pseudo differential operator associated with the `symbol' $a(x, \xi)$. Pseudo differential operators have many applications in the theory of partial differential operators. If the symbol $a(x, \xi)$ is independent 	in $x$ variable, say $a(x, \xi)=m(\xi)$, then the associated operator 
\bes
m(D)f(x)=\int_{\R^n} e^{2\pi i x\cdot\xi} m(\xi)\mathcal F f (\xi)\,d\xi,
\ees
is called  multiplier operator. Boundedness of the multiplier operator is well studied on $\R^n$.	
Let $m:\R^n\rightarrow \C$ satisfies the following H\"ormander-Mihlin differential inequalities:
\begin{equation}\label{eqn:condn_on_multiplier}
\left|\frac{d^j}{d\xi^j} m(\xi)\right|\leq A_j |\xi|^{-j},
\end{equation}
for all $\xi\in\R^n\setminus \{0\}$ and for all $0\leq j\leq [\frac{n}{2}]+1 $. Then  the multiplier operator $m(D)$ is a bounded operator on $L^p(\R^n), 1<p<\infty$.

Now  let $ G $ be a real rank one noncompact connected
semisimple Lie group with finite center. Let $ K $ be a maximal compact subgroup of G,
and let $ \X = G/K  $ be the associated symmetric space.   
Corresponding to a multiplier function $m$, the associated multiplier operator $T_m$ on $\X$ is defined by 
\be\label{multiplier-op-symm}
T_m(f)(g)=\int_\R \int_K m(\lambda) \wtilde{f}(\lambda, k)  e^{-(i\lambda +\rho)H(g^{-1}k)} |c(\lambda)|^{-2}\,d\lambda, \,\,g\in G,
\ee
where $\wtilde{f}$ is the (Helgason) Fourier transform of $f$ and $c(\lambda)$ is the Harish-Chandra $c$-function. The Boundedness of the multiplier operator on the symmetric space  has been studied by Clerc and Stein \cite{Clerc Stein}, Stanton and Tomas \cite{Stanton and Tomas},  Anker and Lohou\'e \cite{Anker-Lohoue}, Taylor \cite{Taylor}, Anker \cite{Anker} and Giulini, Mauceri, Meda \cite{GMM}. Let us elaborate. The pioneer work was done in   \cite{Clerc Stein}. It was observed there that if $ T_m  $ to be a continuous operator on $ L^p(G/K) $, then $ m $ is necessarily   holomorphic and bounded inside the strip $ S_p^\circ $, where   \bes S_p=\left\{\lambda\in\C\mid \,|\Im\lambda|\leq \left|\frac{2}{p}-1\right|\rho\right\} \quad \text{ for  }  p \in [1,\infty]  .
\ees 
Conversely the authors in \cite{Clerc Stein} also gave a sufficient conditions when $ G $ is complex. Similar results were obtained later when $ G $ is real rank one \cite{Stanton and Tomas} and when $ G $ is a normal real form \cite{Anker-Lohoue}.  In 1990,  Anker (\cite{Anker}) improved and generalized  previous results of \cite{Clerc Stein}, \cite{Stanton and Tomas}, \cite{Anker-Lohoue}, and \cite{Taylor} by proving the following multiplier theorem on $\X=G/K$:
\begin{theorem}$($Anker$)$ \label{Anker}
Let $1<p<\infty$, $v=| \frac{1}{p}-\frac{1}{2} |$ and $N=[v \dim \X] +1$. Let $m:\R\rightarrow \C$ extends to an even holomorphic function on $S_p^\circ$, $\frac{\partial^i}{\partial \lambda^i} m    
 \,  (i=0, 1, \cdots, N)$ extends continuously to $S_p$  and satisfies
\bes
\sup_{\lambda\in S_p} (1 + |\lambda|)^{-i}\left|\frac{\partial^i}{\partial \lambda^i} m(\lambda) \right|<\infty,
\ees
for $i=0, 1, \cdots, N$.
Then the associated multiplier operator 
$T_m$ 
is a bounded operator on $L^p(\X)$.
\end{theorem} 
For technical reasons, it was necessary for the author to assume regularity assumptions on the boundary. Nonetheless, Anker mentioned that it should be possible to allow $ m $ to have a singularity at the boundary points $\pm  i\rho_p:= i\left| \frac{2}{p}-1\right|\rho$; and $ m $  will be  still a $ L^p $ multiplier on $\X  $. Later  Ionescu \cite{Ionescu 2000} improved the theorem above by replacing the continuity of the multiplier $m$ on the boundary with that of singularity condition at $\pm i\rho_p$.
 \begin{theorem}$($Ionescu$)$
 Let $1<p<\infty$. Let $m:\R\rightarrow \C$ extends to an even holomorphic function on $S_p^\circ$ and satisfies
\bes
\left|\frac{\partial^j}{\partial \lambda^j} m(\lambda) \right|\leq A_j\left[|\lambda + i\rho_p|^{-j} + |\lambda - i\rho_p|^{-j}\right], \lambda\in S_p,
\ees
for $j=0, 1, \cdots ,\left[\frac{\dim \X+1}{2}\right]+1$.
Then the associated multiplier operator 
$T_m$ 
is a bounded operator on $L^p(\X)$.
\end{theorem}  
 	 
The boundedness of multiplier operator  for fixed $K$-types on $\mathrm{SL}(2, \R)$ has also been studied in \cite{Ricci}. Now one is naturally led to the following questions :  what happens if we replace the multiplier $m(\lambda)$ by a symbol $ \sigma(x,\lambda) $? and if the corresponding operator associated with the symbol should be bounded, then what are the conditions $ \sigma (x,\lambda) $ should satisfy? In this paper we introduce and study the boundedness of pseudo differential operators in the setting of noncompact symmetric spaces of rank one.  In preparation for the statement of our main results, let us recall the analogous theory  of pseudo differential operator in Euclidean setting and also let us introduce some more notations:

Let $\mathcal{S}^m $ be the symbols defined to be the set of all smooth functions $a:\R^n\times \R^n\rightarrow \C$ satisfies,  
		 \begin{equation}\label{eqn:condn_for_symbol_in_eulidean_case}
		 	   \left| {\partial^\beta_x}{\partial _{\xi} ^\alpha} a (x,\xi)\right| \leq  {C_{\alpha,\beta} }\, {( 1+ |\xi| )^{m-|\alpha|}},
		 \end{equation}
	 for all $x, \xi\in\R^n$ and for all multi indices $ \alpha $ and $ \beta $. 
Then we have the following theorem \cite{Stein 93}:
	 \begin{theorem}$($\cite[Chap VI, Theorem 1]{Stein 93}$)$\label{thm:pdo_on_Rn}
	 Let $a\in \mathcal{S}^0$ and  $ a(x, D) $ be the pseudo differential operator associated with the symbol $ a$. Then $a(x, D)$  extends to  a bounded operator on $ L^p(\R^n) $ to itself, for $ 1<p<\infty $.
 	 \end{theorem} 
 	 Moreover, a genralized version of this theorem is true for  a family of symbols $\{a_\tau :\tau \in \Lambda \}$ on $\R^n$:
 	 \begin{theorem}\label{thm:pdo_for_family_of_symbols}
 	 	Suppose $ \{a_\tau :\tau \in \Lambda \}$ be a family of  symbols in $ \mathcal{S}^0 $ such that
 	 	\begin{equation}\label{eqn:condn_on_symbol_family}
 	 		\left| {\partial^\beta_x}{\partial _{\xi} ^\alpha} a _\tau(x,\xi)\right| \leq  {C_{\alpha,\beta} }\, {( 1+ |\xi| )^{-|\alpha|}}, 
 	 	\end{equation} 
 	 	where $ C_{\alpha,\beta} $ does not depend on $ \tau $. Then for $ 1<p<\infty $, there is a constant $ C_p>0  $ independent of $ \tau \in \Lambda $, such that
 	 	\begin{equation}
 	 		\|a_\tau (x, D) (f)\|_{L^p(\R^n)} \leq C_p \|f\|_{L^p(\R^n)},
 	 	\end{equation}
 	 	for all   $ \tau \in \Lambda$ and $ f \in L^p(\R^n) $.
 	 \end{theorem}
 	 Proof of this theorem (Theorem \ref{thm:pdo_for_family_of_symbols}) is exactly similar to  Theorem \ref{thm:pdo_on_Rn} and therefore we omit the proof.
 	 \begin{remark}
 	 	 We would like to mention that for pseudo differential operator  one cannot weaken the regularity assumption on the symbol $ a(x,\xi) $,  having singularity near $ \xi=0 $. Particularly, if the symbol $a  $ satisfies the following simpler looking condition
 	 	\begin{equation}\label{eqn:condn_for_symbol_L2_unbddness}
 	 		\left| {\partial^\beta_x}{\partial _{\xi} ^\alpha} a (x,\xi)\right| \leq  {C_{\alpha,\beta} }\, {|\xi| ^{-|\alpha|}},
 	 	\end{equation}
  	then the corresponding operator may not be bounded. We refer the interested reader to the discussion \cite[page  267]{Stein 93}.
 	 \end{remark}
 
Now suppose $\sigma: \X\times \C\rightarrow \C$ be a suitable function. We define the  pseudo differential operator associated with the symbol $ \sigma $ by
	\begin{align}\label{eqn:defn_of_T_sigmaf_sym_space}
	 \Psi_\sigma f(x)& =  \int_{\R } \int_{K} \sigma(x,\lambda) \tilde{f}(\lambda,k)   e^{-(i\lambda+\rho)H(x^{-1}k)} |c(\lambda)|^{-2} d\lambda \, dk,
	\end{align}
for any smooth compactly supported function $ f $ on $ \X $. Our aim of this paper is to study boundedness of the pseudo differential operator $\Psi_\sigma$ ($\Psi$DO) on rank one symmetric spaces. We divide the study into two cases: when $p=2$ and when $p\not=2$.
First we state our result for the case when $p\not =2$.
 		\begin{theorem}\label{thm:pdo_for_sym_space}
			Suppose  $ p\in (1,2) \cup (2,\infty)$ and $\X=G/K$.  Let  $ \sigma : \X\times  \C \rightarrow \C  $ be a smooth function such that 
			for each $ x\in \X$, $ \lambda \mapsto \sigma(x,\lambda)$ is an even holomorphic function on $ S_p^\circ$ and  satisfies the differential inequalities:
\begin{equation}\label{eqn:hypothesis_of_sigma}
			\left| \frac{\partial^\beta}{{\partial s}^\beta}\frac{\partial^\alpha}{\partial \lambda ^\alpha} \sigma (x a_s,\lambda)\right| \leq  {C_{\alpha,\beta}}{(1+|\lambda| )^{-\alpha}},
		\end{equation}
	for all $ \alpha=0,1,\cdots, \left[\frac{\dim \X+1}{2}\right]+1; \beta=0,1,2 $;  for all $ x \in G, s\in \R$ and $ \lambda \in S_p $.			
Then the operator $ \Psi_\sigma $ extends to a bounded operator on $ L^p{(\mathbb X)} $ to itself.
				\end{theorem}
 
			\begin{remark}
				\begin{enumerate}
					\item Let us compare our result with the Euclidean case.   {As mentioned before in multiplier case, } the holomorphic extension property of the symbol $ \sigma $ is a new and necessary condition  for the pseudo differential operator $ \Psi_{\sigma} $ to be bounded on $ L^p(G/K) $. It is the  boundary behavior of $ \sigma $ which is similar to the Euclidean case. Here we have considered H\"ormander-type conditions near infinity.  
					\item When $ \sigma $ is a multiplier, then one can write $ \Psi_{\sigma} $ as a convolution operator with a $ K $-biinvariant kernel. This is the fundamental difference between the multiplier case and our situation. Particularly, in multiplier theory the author in \cite{Ionescu 2000} used a transference theorem,  which is comparable to the Herz majorizing principle. For higher rank  case  the author used the same principle  to estimate the $ L^p $ norm of multiplier operator (see \cite[Lemma 4.3]{Ionescu 2002}). But in our case one cannot apply the methods above directly, to find the $ L^p $ norm estimate of $ \Psi_{\sigma} $.
					\item We   also have established a connection between the $ L^p $ boundedness of (the local part of ) pseudo differential operator on $G/K  $ with bounded Euclidean pseudo differential operators (see $\S$ \ref{Analysis on the local part}). From the Euclidean counterpart, we derived the derivative condition on the $ ``x"$ variable of $ \sigma(x,\lambda) $.
				\end{enumerate}
			\end{remark}

 Next we consider the case when $p=2$. For $ p=2 $, the boundedness of the multiplier operator  follows easily due to Plancherel theorem in both Euclidean spaces and symmetric spaces. But in  the same way we could not prove $L^2$ boundedness for the pseudo differential operator associated with a symbol.  In this case we get only partial results, in the $K$-biinvariant setting.  Let $ \sigma :\X \times \R \rightarrow \C $ be a smooth  function  such that $ x\mapsto \sigma(x,\lambda) $ is a $ K $-biinvariant function and let $\Psi_\sigma$ be the ``radial'' pseudo differential operator associated with $ \sigma $. That is,
 \begin{equation}\label{eqn:defn_of_Tsigmaf_for_K_biinv_case_real_sym_space}
 	 \Psi_\sigma(f)(x) = \int_{\R}  \sigma(x,\lambda)  \what f (\lambda) \phi_\lambda(x) |c(\lambda)|^{-2} d\lambda,
 \end{equation}
for $ f\in C_c^\infty(G//K) $, where $\what{f}$ is the spherical Fourier transform of the $K$-biinvariant function $f$ and $\phi_\lambda$ is the elementary spherical function. Suppose $\dim \X=d= m_{1}+m_{2}+1$, then we have the following theorem:
\begin{theorem}\label{thm_for_L2_cap_Lp_to_Lp_K_biinv}
	Let  $ \sigma :\X \times \R \rightarrow \C $ be a smooth  function  such that $ x\mapsto \sigma(x,\lambda) $ is $ K $-biinvariant and \bes \sigma(x,\lambda) =\sigma(x,-\lambda)  \text{ for all }  \lambda \in \R. 
 \ees  Assume that $ \sigma $ satisfies the following  differential inequalities: 
\begin{equation}\label{eqn:condn_of_sigma_for_real_rank_one_L2 case}
	\left|\frac{\partial^\beta}{{\partial t}^\beta} \, \frac{\partial^\alpha}{{\partial \lambda }^\alpha } \sigma(a_t,\lambda) \right| \leq C_{\alpha,\beta} (1+|\lambda|)^{-\alpha-d-1},
\end{equation}
for all $ \lambda , t \in \R  $ and  $ \alpha,\beta=0, 1, 2$. Then the operator $ \Psi_\sigma $ extends to a bounded operator from $ L^2(G//K)\cap L^p(G//K) $ to $ L^2{(G//K)} $ for any $ 1\leq p<2 $.
\end{theorem}

We improve this result in the following subcases:

\begin{theorem}\label{thm:pdo_real_sym_space_K_biinv_2<p<3}
Let $ \frac{3}{2}<p<2 $ and  $\X=G/K$. Let  $ \sigma :\X \times \R \rightarrow \C $ be a smooth  function  such that $ x\mapsto \sigma(x,\lambda) $ is $ K $-biinvariant and satisfies \begin{equation}\nonumber
			\left| \frac{\partial^\beta}{{\partial s}^\beta}\frac{\partial^\alpha}{\partial \lambda ^\alpha} \sigma (xa_s,\lambda)\right| \leq  {C_{\alpha,\beta}}{(1+|\lambda| )^{-\alpha-d-1}},
		\end{equation}
	for all $ \alpha=0,1,\cdots ,\dim \X+2$; $\beta=0,1,2$; for all $ x\in G, s\in \R$ and $ \lambda \in S_p $.	 Then the operator $ \Psi_\sigma $ extends to a bounded operator from $ L^p(G//K)$ to $ L^2{(G//K)} $.
\end{theorem}
We have the following drawbacks in the last two theorems:
\begin{enumerate}
\item We had to take more decay condition on the symbol $\sigma$ than the required $\mathcal S^0$ condition.
\item Also we did not get the required $L^2-L^2$ bound.
\end{enumerate}  
 All of these drawbacks are due to technical difficulties. We could not use direct semisimple theory to the proof; instead, we prove the theorems using boundedness of Euclidean pseudo differential operator. We  overcome these drawbacks in complex symmetric spaces (in the $K$-biinvariant case). The following theorem is an analogue of 
 Calder\'on and Vaillancourt theorem \cite[Theorem 5.1]{Tosio76}, where the assumption on the symbol is weaker than usual $S^0$ condition.
 
 \begin{theorem}\label{thm:pdo_complex-case}
 Let $\X$ be an arbitrary rank complex symmetric space.
Let $ \sigma : \X \times \mathfrak{a}^*  \rightarrow \C$ be a $C^\infty$ function, such that $x\mapsto \sigma(x, \lambda)$ is $ K $-biinvariant, $ \lambda \rightarrow \sigma(x,\lambda)$ is $ W $-invariant and satisfies the following inequalities:
\begin{equation}\label{eqn:condn_of_sigma_for_complex_higher_rank}
	\left|\frac{\partial^\beta}{{\partial H}^\beta} \, \frac{\partial^\alpha}{{\partial \lambda}^\alpha} \sigma(H,\lambda) \right| \leq C_{\alpha,\beta} 
\end{equation}	
 for  all  $\lambda \in  \mathfrak{a}^*, H\in \mathfrak{a}$ and for all multi index $\alpha, \beta$ \text{with $|\alpha|,|\beta|  \leq [\text{rank of } \X/2]+1$}.
Then the pseudo differential operator $ \Psi_\sigma $ extends to a bounded operator on $ L^2{(G//K)} $.
\end{theorem}

For the boundedness of the pseudo differential operators  on nilpotent Lie groups and stratified Lie groups, we refer to \cite{BT, EPP, FR} and references there in.

 Layout of this paper is as follows: Section \ref{Preliminaries} sets the harmonic analysis background on semisimple Lie groups and symmetric spaces. In Section \ref{decomposition of PSi},  we will write $ \Psi_{\sigma} $ as singular integral operators. We will prove Theorem \ref{thm:pdo_for_sym_space} in Sections \ref{Analysis on the local part} and \ref{analysis on global part}.  We prove Theorem \ref{thm_for_L2_cap_Lp_to_Lp_K_biinv}, Theorem \ref{thm:pdo_real_sym_space_K_biinv_2<p<3} and Theorem \ref{thm:pdo_complex-case} in Section \ref{L2 theorems proof}.
In Appendix we state Coifman-Weiss transference theorem.

			\section{\textbf{Preliminaries}} \label{Preliminaries}In this section, we describe the necessary preliminaries regarding semisimple Lie groups and harmonic analysis on Riemannian symmetric spaces. These are standard and can be found, for example, in \cite{GV, H2, Helgason GA}. To make the article self-contained, we shall gather only those results  used throughout this paper.
Let $ G $ be a  noncompact connected semisimple  real rank one Lie group with finite center, with its Lie algebra $\mathfrak g$. Let $ \theta $ be a Cartan involution of $ \mathfrak{g} $ and $ \mathfrak{g} =\mathfrak{k+p} $ be the associated Cartan decompostion. Let $ K=\exp \mathfrak{k} $ be a maximal compact subgroup of $ G $ and let $ \X =G/K $ be an associated symmetric space with origin $ \textbf{0} =\{K\} $.  Let $ \mathfrak{a} $ be a maximal abelian subspace of $ \mathfrak{p} $. Since the group $ G $ is of real rank one,  dim $\mathfrak{a} =1 $.  Let $ \Sigma $ be the set of nonzero roots of the pair $ (\mathfrak{g,a}) $, and let $ W $ be the associated Weyl group. For rank one case, it is well known that  either $ \Sigma =\{-\alpha,\alpha\} $ or $ \{-2\alpha,-\alpha,\alpha,2\alpha \} $, where $ \alpha $ is a positive root  and the Weyl group $ W $ associated to $ \Sigma  $ is  \{-Id, Id\}, where Id is the identity operator. Let $ \mathfrak{a}^+ =\{ H \in \mathfrak{a} : \alpha(H)>0 \} $ be a positive Weyl chamber, and let $ \Sigma^+ $ be the corresponding set of positive roots.  In our case, $ \Sigma ^+ = \{\alpha\}$ or $ \{ \alpha, 2\alpha\} $. For any root $ \beta \in \Sigma  $, let $ \mathfrak{g}_\beta $ be the root space associated to $ \beta $. Let 
			\bes 
			\mathfrak{n} =\sum_{\beta \in \Sigma^+} \mathfrak{g}_\beta, 
			\ees and let \bes \ol { \mf{n}} =\theta(\mf{n} ).\ees  Also let 
			\bes  N =\exp  \mathfrak{n},   \ol N =\exp \ol{ \mathfrak{n}}.
			\ees
The group $ G $ has an Iwasawa decomposition \bes G= K (\exp \mathfrak{a})N,
			\ees and a Cartan decomposition 
			\bes G=K(\exp \mathfrak{a}^+) K.\ees This decomposition is unique. For each $ g \in G $, we denote $ H(g) \in \mathfrak{a} $
			and $ g^+ \in \ol {\mf{a}^+} $ are the unique elements such that
			\be\label{defn:Hg}
			g=k \exp H(g) n, k\in K, n\in N,
			\ee  and 
\be\label{defn:gplus}
g=k_1 \exp (g^+) k_2, k_1, k_2\in K.
\ee		
We also have another Iwasawa  decomposition \bes G=\ol N (\exp \mf{a})K. \ees
Let  $H_0$ be the unique element in $\mathfrak a$ such that $\alpha(H_0)=1$ and through this we identify $\mathfrak a$ with $\R$ as $t\leftrightarrow tH_0$ and $\mathfrak a_+= \{H\in \mathfrak a\mid \alpha(H)>0\}$ is identified with the set of positive real numbers.   We also identify $\mathfrak a^*$ and its complexification $\mathfrak a^*_\C$ with $\R$  and $\C$ respectively by $t\leftrightarrow t\alpha$ and  $z\leftrightarrow z\alpha$, $t\in \R$, $z\in \C$. Let $A=\exp\mathfrak{a}=\left\{a_t:=\exp(t H_0)\mid t\in\R\right\}$ and $A^+=\left\{a_t\mid t>0\right\}$.
 Let $m_1=\dim \mathfrak g_\alpha$ and $m_{2}=\dim \mathfrak g_{2\alpha}$ where $\mathfrak g_\alpha$ and
 $\mathfrak g_{2\alpha}$ are the root spaces corresponding to $\alpha$ and $2\alpha$. As usual then $\rho=\frac 12(m_1+2m_2)\alpha$ denotes the half sum of the positive roots.
By abuse of notation we will denote $\rho(H_0)=\frac 12(m_1+2m_2)$ by $\rho$.

Let $ dg, dk,  dn $ and $d\ol{n}$ be the Haar measures on the groups $ G, K, N$ and $\ol{N}$ respectively. We  normalize $ dk $ such that $ \int_K dk =1 $.  We have the following integral formulae corresponding to the Iwasawa and Cartan
decomposition respectively, which holds for any integrable function $f$:

\begin{equation}\label{eqn:integral-decom-nbar-a}
				\int_{G}f(g) dg =\int_{\ol N} \int_{\R} \int_{K} f(\ol  n a_t k)e^{2\rho t} dk dt d\ol n,
			\end{equation}
	and		
			\begin{equation}
\int_Gf(g)dg=\int_K\int_{\R^+}\int_K f(k_1a_tk_2) \Delta(t)\,dk_1\,dt\,dk_2. \label{polar}
\end{equation} 
where $\Delta(t)=(2\sinh t)^{m_1 + m_2}(2\cosh t)^{m_2}$.

\subsection{Iwasawa meets Cartan:} We will use the Iwasawa decomposition of $ G =\ol NAK $ and its connection with the Cartan decomposition. This idea was also used by \cite{Stromberg} and later in this direction   Ionescu proved the following: 

\begin{lemma}$($\cite[Lemma 3]{Ionescu 2002}$)$\label{lemma:expression_of_v-a _r}
	If $ \ol v \in \ol N $ and $ r\geq 0 $, then 
	\begin{equation}\label{eqn:expression_of_[va_r]}
		[\ol v a_r]^+ = r +H(\ol v)+ E(\ol v, r), 
	\end{equation}
	where \begin{equation}
		0\leq E(\ol v, r)\leq 2 e^{-2r}.
	\end{equation}
\end{lemma}
Let $ P(\ol v) = e^{-\rho (H(\ol v))} $;  then for any $ \epsilon_0 >0 $ we have (see \cite[(3.11)]{Ionescu 2002})
\begin{equation}\label{eqn:L1_norm_P(v-) (1+epsilon)_is_finite}
	\int_{\ol N} P(\ol v)^{1+\epsilon_0} d\ol v =C_{\epsilon_0 }<\infty.
\end{equation}
Also we have (see \cite[Theorem 6.1, Chapter II, p. 180]{Helgason GA})  
\bes
e^{2\alpha (H(\ol v))} = \left[ 1+c_0 |X|^2\right]^2+ 4 c_0 |Y|^2 ,
\ees
where $ X $ and $ Y $ are the coordinates of $ \ol v $ in $ \ol N $ corresponding to the root spaces $ \mathfrak{g}_{-\alpha} $ and $ \mathfrak{g}_{-2\alpha} $, and $ c_0^{-1}=\frac{1}{4} (m_\alpha + m_{2\alpha}) $. Therefore,
\be
\label{eqn:H(v)-positive}
H(\ol v) \geq 0,  \,\, \text{for all } \ol v \in \ol N.
\ee

We recall the abelian group $ A $ acts as a dilation on $ \ol N  $, by the mapping
\begin{eqnarray}
	\ol n\rightarrow a \ol n a^{-1} \in \ol N.
\end{eqnarray}
Moreover if  $\delta_r$ be a dilation of $\ol N$, defined by
\bes
\delta_r(\ol n):= \exp (r H_0)\ol n \exp(-r H_0),
\ees 
then the following is true
\begin{equation}\label{eqn:dialation_formula_for_N-}
	\int_{\ol N} \mathfrak{h}(\delta_r(\ol n)) \,d\ol n =  e^{2\rho r} \int_{\ol N} \mathfrak{h}( \ol n)  \,d\ol n,
\end{equation}
for any integrable function $ \mathfrak{h} $ on $ \ol N $.


\subsection{Fourier transform}For a sufficiently nice function $f$ on $\X$, its (Helgason) Fourier transform $\widetilde{f}$ is a function defined on $\C \times K$ given by 
\be \label{defn:hft}
\widetilde{f}(\lambda,k) = \int_{G} f(g) e^{(i\lambda- \rho)H(g^{-1}k)} dg,\:\:\:\:\:\: \lambda \in \C,\:\: k \in K, 
\ee
whenever the integral exists (\cite[p. 199]{Helgason GA}). 

It is known that if $f\in L^1(\X)$ then $\widetilde{f}(\lambda, k)$ is a continuous function of $\lambda \in \R$, for almost every $k\in K$. If in addition $\widetilde{f}\in L^1(\R\times K, |c(\lambda)|^{-2}~d\lambda~dk)$ then the following Fourier inversion holds,
\be\label{hft}
f(gK)= |W|^{-1}\int_{\R\times K}\widetilde{f}(\lambda, k)~e^{-(i\lambda+\rho)H(g^{-1}k)} ~ |c(\lambda)|^{-2}d\lambda~dk,
\ee
for almost every $gK\in \X$ (\cite[Chapter III, Theorem 1.8, Theorem 1.9]{Helgason GA}), where $c(\lambda)$ is the Harish Chandra's $c$-function given by \bes c(\lambda)=\frac{2^{\rho-i\lambda} \Gamma(\frac{m_1 + m_2 +1}{2})\Gamma(i\lambda)}{\Gamma(\frac{\rho+ i\lambda}{2}) \Gamma(\frac{m_1+ 2}{4} + \frac{i\lambda}{2})}.\ees
It is normalized such that $c(-i\rho)=1$.

 Moreover, $f \mapsto \widetilde{f}$ extends to an isometry of $L^2(\X)$ onto $L^2(\R\times K, |c(\lambda)|^{-2}~d\lambda~dk )$ (\cite[Chapter III, Theorem 1.5]{Helgason GA}).

A function $f$ is called $K$-biinvariant if \bes f(k_1xk_2)=f(x) \text{ for all } x\in G, k_1, k_2\in K.\ees  We denote the set of all $K$-biinvariant functions by $\mathcal F (G//K)$.
Let $\mathbb D(G/K)$ be the algebra of $G$-invariant differential operators on $G/K$. The elementary spherical functions $\phi$ are $C^\infty$ functions and are joint eigenfunctions of all $D\in\mathbb D(G/K)$ for some complex eigenvalue $\lambda(D)$. That is $$D\phi=\lambda(D)\phi, D\in\mathbb D(G/K).$$They are parametrized by $\lambda\in\C$. The algebra $\mathbb D(G/K)$ is generated by the Laplace-Beltrami operator $ L$. Then we have, for all $\lambda\in\C, \phi_\lambda$ is a $C^\infty$ solution of  
\begin{equation}\label{phi-lambda}
L\phi=-(\lambda^2 + \rho^2)\phi.
\end{equation}
			For any $\lambda\in \C$  the elementary spherical function $\phi_\lambda$ has the following integral representation
$$\phi_\lambda(x)=\int_K e^{-(i\lambda+\rho)H(xk)}\,dk \text{ for all } x\in G.$$
The spherical transform $\what{f}$ of a  suitable $K$-biinvariant function $f$ is defined by the formula:
$$\what{f}(\lambda)=\int_Gf(x)\phi_\lambda(x^{-1})\,dx.$$
It is easy to check that for suitable $K$-biinvariant function $f$ on $G$, its (Helgason) Fourier transform $\widetilde{f}$ reduces to the spherical transform $\what{f}$.

We now list down some well known properties of the elementary spherical functions which are important for us (\cite[ Prop 3.1.4 and Chapter 4, \S 4.6]{GV}, \cite[Lemma 1.18, p. 221]{Helgason GA}).

\begin{enumerate}
\item[(1)] $\phi_\lambda(g)$ is $K$-biinvariant in $g\in G$,  $\phi_\lambda=\phi_{-\lambda}$, $\phi_\lambda(g)=\phi_\lambda(g^{-1})$.
\item[(2)] $\phi_\lambda(g)$ is $C^\infty$ in $g\in G$ and holomorphic in $\lambda\in\C$.
\item[(3)] The following inequality holds:
\bes
e^{-\rho t} \leq \phi_0(a_t)\leq \left(1+|t|\right)~e^{-\rho t}, \:\: t\geq 0.
\ees
\item[(4)] $|\phi_\lambda(x)|\leq 1$ for all $x\in G$ if and only if $\lambda\in S_1=\left\{\lambda\in \C \mid |\Im\lambda|\leq \rho\right\}$.
\item[(5)] For all $\lambda\in \R$ we have
\bes
|\phi_\lambda(g)| \leq  \phi_0(g)\leq 1.
\ees
\item [(6)]The function $\phi_\lambda$ satisfies the following identity (\cite[Chapter 4, Lemma 4.4, p. 418]{H2}):
\be\label{eqn:identity-phi}
\phi_\lambda(y^{-1}x)=\int_K e^{-(i\lambda +\rho)H(x^{-1}k)} e^{(i\lambda-\rho)H(y^{-1}k)}\,dk,
\ee
for all $x, y\in G$ and $\lambda\in\R$.
\end{enumerate}		

Let $C_c^\infty(G//K)$ be the set of all $C^\infty$ compactly supported $K$-biinvariant functions on $G$.  Also let $PW(\C)$ be the set of all entire functions $h:\C\rightarrow \C$ such that $ h $ is of exponential type $ T $ for some $ T > 0 $, that is, for each $N\in \N$, $$\sup_{\lambda\in\C}\,(1 +|\lambda|)^N |h(\lambda)| e^{-T |\Im\lambda|}<\infty$$   and let $PW(\C)_e$ be the set of all even functions in $PW(\C)$. 

Then we have the following Paley-Wiener  theorem:

\begin{theorem}\label{Paley-Wiener}
The function $f\mapsto \what{f}$ is a topological isomorphism between $C_c^\infty(G//K)$ and $PW(\C)_e$. 
\end{theorem}

The Abel transform of a suitable $K$-biinvariant function is defined by
\begin{equation*}
	\mathcal{A}(f)(t)= e^{\rho t}\int_{\ol N} f(\ol v a_t)\,  d\ol v,
\end{equation*}
whenever the integral exists.  It satisfies the following slice projection theorem: \be \label{prop:abelslice}
\mathcal F\mathcal A f(\lambda)=\widehat{f}(\lambda),
\ee
for all $\lambda\in\C$, for which both side exists. From this, it follows that for a  $ K $-biinvariant function $f$, \be\label{prop:abel-even}
\mathcal A f(t)=\mathcal A f(-t),
\ee for all $t\in\C$ for which Abel transform of $f$ exists.		
We need the following mapping property of the Abel transform \cite{RS}:
\begin{theorem}[{S. K. Ray and R.P. Sarkar}]\label{thm:Ray-Sarkar}
	Suppose $ 1\leq p<2 $ and $ \gamma_p =\frac{2}{p}-1 $. Then 
	\begin{equation}\label{eqn:result_of_Sk_ray_and_RP_sarkar}
		\left(  \int_{\R} |\mathcal{A}f(a_t)|^r e^{\rho \beta |t|} dt \right)^{\frac{1}{r}} \leq C\| f\|_{L^p(\X)}  \, \text{ for } 1\leq  r< \frac{1}{\gamma_p} , 0<\beta <\gamma_p r.
	\end{equation}
\end{theorem}
Putting $r=p$ in (\ref{eqn:result_of_Sk_ray_and_RP_sarkar}) for $1<p<2$ we get
\bes
\left(  \int_{\R} |\mathcal{A}f(a_t)|^p dt \right)^{\frac{1}{p}} \leq C\| f\|_{L^p(\X)}.
\ees


	\subsection{Spherical function} In our study of  $ L^p $ boundedness of $ \Psi_{\sigma} $, representation  of spherical function $ \phi_{\lambda} $ is crucial. Although, such formulations   of   spherical function are well known and ubiquitous in the literature, for completeness we briefly recall the explicit expression of $ \phi_{\lambda} (a_t)$ for small and large values of $ t $. 
	
	 Let $J_\mu(z)$ be the Bessel functions of first kind and 
		 \bes
		 \mathcal J_\mu(z)=\frac{J_\mu(z)}{z^\mu} \Gamma(\mu + \frac{1}{2}) \Gamma(\frac{1}{2}) 2^{\mu-1},
		 \ees
		 and \bes
		 c_0=\pi^{\frac{1}{2}} 2^{\frac{m_1+2m_{2}}{2}-2}\frac{\Gamma(\frac{d-1}{2})}{\Gamma(\frac{d}{2})}.
		 \ees
		 We have the following expansion of $ \phi_{\lambda} $ near origin:
			\begin{theorem}$($\cite[Theorem 2.1]{Stanton and Tomas}$)$\label{thm:localexpansion-phi}
				There exist $ R_0>1 $ and $ R_1>1  $ such that for any $ t $ with $ 0\leq t\leq R_0 $ and any $ M\geq 0 $,
				\begin{eqnarray}\nonumber
					\phi_{\lambda}(a_t) =c_0 \left[ \frac{t^{d-1}}{\Delta(t)}\right]^{\frac{1}{2}} \sum_{m=0}^{\infty} t^{2m} a_m(t)  \mathcal{J}_{(d-2)/2 +m}(\lambda t)	 	 \\
				\nonumber	\phi_{\lambda}(a_t) =c_0 \left[ \frac{t^{d-1}}{\Delta(t)}\right]^{\frac{1}{2}} \sum_{m=0}^{M} t^{2m} a_m(t)  \mathcal{J}_{(d-2)/2 +m}(\lambda t)	+E_{M+1}(\lambda t) 	  		 	 
				\end{eqnarray}
				where \begin{align}\nonumber
					a_0(t) &=1\\ \notag
					|a_m(t)|& \leq C R_1^{-m}
				\end{align}
				\begin{align} \nonumber 
				|E_{M+1}(\lambda t)| &\leq C_M \,t^{2(M+1)} \hspace{3.5cm} \text{ if }|\lambda t|\leq 1\\
					&\leq \nonumber C_M \, t^{2(M+1)} (\lambda t )^{-((d-1)/2 +M+1)} \hspace*{.3cm} \text{ if }|\lambda t|>1.
				\end{align}
			\end{theorem}			
		 
Next we have the following expansion of $ \phi_{\lambda} $ away from origin.
			
			\begin{proposition} $($\cite[Prop. A1]{Ionescu 2002}$)$\label{propn:global_expansion_of_phi_lambda}
				If $ t\geq 1/10 $ and $ N\in\N $, then $ \phi_\lambda(a_t) $ can be written in the following form
				\begin{equation}\label{eqn:phi-lambda-assymp-expansion}
					\phi_\lambda(a_t)= e^{-\rho t} \left( e^{i\lambda t}c(\lambda)\left(1+a(\lambda,t)\right) +e^{-i\lambda t} c(-\lambda) \left(1+a(-\lambda,t)\right) \right),
				\end{equation}
				where the function $ a(\lambda,t) $ satisfies the following inequalities,
				\begin{equation}\label{eqn:est_for_a(lambda,r)}
					\left|  \frac{\partial^\alpha}{\partial \lambda^\alpha} \frac{\partial^l}{\partial t^l} a(\lambda,t) \right| \leq C (1+|\Re \lambda | )^{-\alpha},
				\end{equation}
				for all integers $ \alpha\in [0,N] $, $ l\in \{0,1\} $, and for all $ \lambda  $ in the region $ 0\leq \Im \lambda \leq  \rho+1/10. $
			\end{proposition}

The following two lemmas give the estimates of $c(-\lambda)^{-1}$ and $|c(-\lambda)|^{-2}$, along with their derivatives.		
\begin{lemma}$($\cite[App. A]{Ionescu 2000}$)$
The function $ \lambda \rightarrow c(-\lambda)^{-1} $ is holomorphic inside the region $ \Im \lambda\geq 0 $ and satisfies the estimates
			
			\begin{equation}\label{est: cminusla}
				\left|	\frac{d^\alpha}{d\lambda^\alpha} c(-\lambda)^{-1}\right | \leq C_\alpha(1+|\lambda|)^{\frac{d-1}{2}-\alpha},
			\end{equation}
			for all integers $ \alpha = 0,1,... N$ and for all $ \lambda  $ with  $ 0\leq \Im \lambda \leq \rho $.
\end{lemma}

			\begin{lemma} $($ \cite[Lemma 4.2]{Stanton and Tomas}$)$ \label{lemma:est_of_|c(-lambda)|^-2}
				We have \begin{equation}\label{estimate of |c(-lambda)|^-2}
					\left|	\frac{d^\alpha}{d\lambda^\alpha} |c(-\lambda)|^{-2}\right | \leq C_\alpha(1+|\lambda|)^{d-1-\alpha},
				\end{equation} 
				for all $ \lambda \in \R $ and $ \alpha \in [0,N] $. 
			\end{lemma}
		Later we will require the following lemma.
		
		\begin{lemma}\label{lemma:estimate_of_sigma_(lambda)_(c(-lambda)_e(-epsilo^2_lamda))}
			Let $ \sigma $ satisfies \eqref{eqn:hypothesis_of_sigma} and $ N\in \N $, then the following is true  for all $ j\in [0,N]  $,
			\begin{equation}\nonumber
				\lim_{\epsilon \rightarrow 0} \left| \frac{\partial^j}{\partial \lambda^j} \left( \frac{\sigma(x,\lambda)\, e^{-\epsilon \lambda^2}}{c(-\lambda)}\right)\right| \leq \frac{C_j}{(1+|\lambda|)^{j-\frac{(d-1)}{2}}} \text{ for all $ \lambda $ with $ 0 \leq \Im \lambda \leq \rho_p $},
			\end{equation}
		where $ C_j $'s are independent of $ x\in \X. $
		\end{lemma}
		\begin{proof}
			Let $ 0<\epsilon\leq 1$. By a  simple calculation we have,  
			\begin{equation}\label{eqn:est_of_lambda^ne^-elambda^2}
				|\lambda^n e^{-\epsilon \lambda^2}| \leq \left( \left(\frac{n}{2}\right)^{\frac{n}{2}}e^{-\frac{n}{2}} \right)\epsilon^{-\frac{n}{2}} \text{ for all }  \lambda \in \R, n\in \N.
			\end{equation}
			Then for $ j\in [0,N] $ and $ x\in \X $,   
			\begin{align*}
				\left| \frac{\partial^j}{\partial \lambda^j} \left( \frac{\sigma(x,\lambda)\, e^{-\epsilon \lambda^2}}{c(-\lambda)}\right)\right|\leq C \sum_{l=0}^{j}     \left|  \frac{\partial^l }{\partial \lambda^l} \left( {\sigma(x,\lambda)\, e^{-\epsilon \lambda^2}}\right) \right| \left|  \frac{\partial^{j-l} }{\partial \lambda^{j-l}}  c(-\lambda)^{-1} \right|.
			\end{align*}
			Now we apply  \eqref{eqn:hypothesis_of_sigma}, (\ref{est: cminusla}) and \eqref{eqn:est_of_lambda^ne^-elambda^2} repeatedly to get,
			\begin{align*}
				\sum_{l=0}^{j}     \left|  \frac{\partial^l }{\partial \lambda^l} \left( {\sigma(x,\lambda)\, e^{-\epsilon \lambda^2}}\right) \right| \left|  \frac{\partial^{j-l} }{\partial \lambda^{j-l}} c(-\lambda)^{-1} \right| \leq \frac{1}{(1+\lambda)^{j-\frac{d-1}{2}}} \left(  1+\epsilon^{\frac{1}{2}} (1+|\lambda|)  + \cdots +\epsilon^{\frac{j}{2}} (1+|\lambda|)^j \right).
			\end{align*}
			By sending $ \epsilon \rightarrow 0 $ both side, we get our lemma.
		\end{proof}


\section{\textbf{Singular integral realization of $ \Psi $DO}}\label{decomposition of PSi}
 We now study the integral representations of the pseudo differnetail operator $\Psi_\sigma $.  By substituting the definition of  Helgason Fourier transform (\ref{defn:hft}) in \eqref{eqn:defn_of_T_sigmaf_sym_space}, it transforms to
		\begin{align*}\label{eqn:alternate_defn_of_T_sigmaf_sym_space}
			\Psi_\sigma f(x)& =  \int_{\R } \int_{K} \sigma(x,\lambda) \tilde{f}(\lambda,k)   e^{-(i\lambda+\rho)H(x^{-1}k)} |c(\lambda)|^{-2} d\lambda \, dk\\
			&=  \int_{G}  \int_{\R } \int_{K}  f(y)\, \sigma(x,\lambda) \, e^{-(i\lambda+\rho)H(x^{-1}k)}e^{(i\lambda -\rho) H(y^{-1}k)}  |c(\lambda)|^{-2}  dk\, d\lambda\,   dy.
		\end{align*}
			Then using a formula of the spherical function (\ref{eqn:identity-phi}),  we have
			\begin{equation}
				\Psi_\sigma f(x) =  \int_{G} f(y) \mathcal{K}(x,y) \,dy,
			\end{equation}
where
\begin{equation}\label{eqn:defn_of_K(x,y)}
				 \mathcal{K}(x,y) = \int_{\R} \sigma (x,\lambda) \phi_\lambda(y^{-1}x ) |c(\lambda)|^{-2} d\lambda.
			\end{equation}		

 \textbf{Decomposition of the $\Psi$DO : } Due to distinct behavior of spherical function near the origin and away from it, we now split $ \Psi_{\sigma} $ into two parts. The idea of such decomposition employed here was first introduced by Clerc and Stein \cite{Clerc Stein}, where they proved multiplier theorem for complex $ G $. Suppose $ \eta: \R \rightarrow [0,1] $ be a smooth even function supported on $ [1,\infty) $ such that $ \eta(t) =1 $ if $ |t| \geq 2 $. Let $  \eta^\circ(t) =1-\eta(t)$. Using the Cartan decomposition, we extend the functions $ \eta $ and $ \eta^\circ $ to  $ K $-biinvariant functions on $ G $, by
			\begin{equation}
				\eta(x)= \eta(x^+) \quad \text{ for all } x\in G. 
			\end{equation} 
		Now we decompose the  operator $\Psi_\sigma   $ as a sum of local and global parts as:
			\begin{equation}\label{eqn:decomp_of_the_main_kernel}
				 \Psi_\sigma f(x) = \Psi_\sigma^{loc}f(x)+ \Psi_\sigma^{glo}f(x) ,
			 \end{equation}
		 where 
		 \begin{equation}\label{eqn:defn_of_T_sigma_local}
		\Psi_\sigma^{loc}f(x)=	\int_{G}  f(y) \eta^\circ(y^{-1} x) \mathcal{K}(x,y) dy,
		 \end {equation} 
		 and 
		 \be\label{eqn:defn_of_T_sigma_global}
		 \Psi_\sigma^{glo}f(x)  =\int_{G}  f(y)  \eta(y^{-1} x) \mathcal{K}(x,y) dy.
		 \ee
In the upcoming two sections, we   separately  deal with the operators $ \Psi^{loc}_\sigma $ and $ \Psi_\sigma^{glo} $.
\section{\textbf{Analysis on the local part of $ \Psi_{\sigma} $}}\label{Analysis on the local part}
In the multiplier case, one can write the local part of the multiplier operator as convolution with a compactly supported $ K $-biinvariant function $ \mathcal{K}_m $ (say). Next they relate the convolution with $ \mathcal{K}_m $ to Euclidean multiplier and then the boundedness of multiplier on $ G/K $ follows from Marcinkiewicz multiplier theorem. Here, we point out a fundamental difference between the multiplier case and our situation.  In the case of multiplier operator, writing it as a convolution operator plays a crucial role. However in pseudo differential operator, one cannot use the theory of multiplier, due to an extra variable $ x $ in the symbol $ \sigma(x,\lambda) $.  To prove the boundedness of $ \Psi_{\sigma}^{loc} $, we used a generalized transference principle by Coifmann-Weiss. This principle helped us to establish a connection between the $ L^p $ boundedness of $  \Psi_{\sigma}^{loc}$ on $ G/K $ with Euclidean pseudo differential operator. We have the following result for the local part of $\Psi_\sigma$.

	\begin{theorem}
	\label{thm:Lp_boundedness_of_T_sigma_local}
		Suppose $ p\in  (1,\infty) $ and $\sigma$ satisfies the properties of Theorem \ref{thm:pdo_for_sym_space}.  Then there is a constant $ C> 0 $, such that  \begin{equation}\label{eqn Lp estimate of T sigma f }
			\| \Psi^{loc}_\sigma f\|_{L^p(\X)} \leq C\|f\|_{L^{p}(\X)}, \quad \text{ for all } f\in L^p(\X).
		\end{equation}
	\end{theorem}  

Proof of the theorem above will be a consequence of the following lemma. In fact next lemma allow us to apply Coifman-Weiss transference principle (see \eqref{eqn:req_est_for_pi}) in our setting. We will postpone the proof of following lemma   in the next subsection. Assuming this, we complete the proof of Theorem  \ref{thm:Lp_boundedness_of_T_sigma_local}.
\begin{lemma}\label{lemma:Coifman-Weiss_condition_our_case}
	 For $ z\in G $, let   $\mathcal R_s z:=za_s$, $ s\in \R $. Then we have the following inequality
	\begin{equation}\label{eqn:Coifman-Weiss_condn_satisified_by_K}
		\left(	\int_{\R} \left| \int_{\R} \mathcal{K}(\mathcal{R}_s z, \mathcal{R}_{-t} \mathcal{R}_s z ) \eta^\circ(a_t) |\Delta(t)|  g(s-t) dt\right|^{p} ds\right) ^{\frac{1}{p}}\leq C \left( \int_{\R}|g(t)|^p dt\right)^{\frac{1}{p}},
	\end{equation}
 for some $ C>0 $ is independent of $ z \in G$, and $ g\in L^p(\R) $.
\end{lemma}

\begin{proof}[\textbf{Proof of Theorem \ref{thm:Lp_boundedness_of_T_sigma_local} assuming Lemma \ref{lemma:Coifman-Weiss_condition_our_case}:}] We will proceed with integral formula corresponding to Cartan decomposition  followed by Minkowski’s integral inequality to obtain
	 
	\begin{align*} 
		\left( \int_{\X} |\Psi_\sigma^{loc} f(x) | ^p dx\right)^{\frac{1}{p}} &= \left( \int_{G} \left|  \int_{G} f(y) \eta^\circ(y^{-1}x) \mathcal{K}(x,y) dy  \right|^{p}  dx \right)^{\frac{1}{p}}\\
		& = \left( \int_{G} \left|  \int_{G} f(xy) \eta^\circ(y^{-1}) \mathcal{K}(x,xy) dy  \right|^{p}dx  \right)^{\frac{1}{p}}\\
		& = \left( \int_{G} \left| \int_{K} \int_{\R} f(xka_t) \eta^\circ\left( (ka_t)^{-1} \right) \mathcal{K}(x, xk a_t) |\Delta(t)| dk dt    \right|^p dx \right)^{\frac{1}{p}}\\
		&\leq \int_{K} \left( \int_{G} \left|  \int_{\R} f(xka_t) \eta^\circ\left(a_t\right) \mathcal{K}(x, xk a_t) |\Delta(t)|  dt    \right|^p dx\right)^{\frac{1}{p}}dk\\
		& = \left( \int_{G}  \left|  \int_{\R} f(za_t) \eta^\circ\left(a_t\right) \mathcal{K}(z, z a_t) |\Delta(t)|  dt    \right|^p dz\right)^{\frac{1}{p}}\\
		& =  \left( \int_{G}  \left|  \int_{\R} \mathcal{R}_{t} f(z) \,   \eta^\circ( a_t) \mathcal{K}(z, z a_t) |\Delta(t)|  dt    \right|^p dz\right)^{\frac{1}{p}},
	\end{align*}		
	where $\mathcal{R} $  is the representation of $ \R  $, acting on the functions of $ G $, by 
	\bes \mathcal{R}_t f(x) =f(xa_t ).
	\ees 	
Now we apply the generalized Coifman-Weiss  transference principle (see Appendix \eqref{eqn:reprn_of_T_using_K}) and Lemma \ref{lemma:Coifman-Weiss_condition_our_case}, to conclude the theorem. 
\end{proof}
	

	\subsection{\text{Proof of Lemma \ref{lemma:Coifman-Weiss_condition_our_case}}}
	Before proceeding further, we remark that  we assume that $ \sigma(\cdot ,\lambda ) $ has rapid decay as $ \lambda \rightarrow \infty $.  To be precise, we  assume $ \sigma(\cdot, \lambda) $ is multiplied with a factor of the form $ e^{-\epsilon \lambda^2} ,$ where $ 0<\epsilon \leq 1 $.  Our estimates are uniform in $ \epsilon $. Once we prove suitable uniform estimates, standard limiting arguments allow us to pass to the general case. 
	Let $\mathcal{K}_1:G\times \R^+\rightarrow \C$ be a function defined by \begin{align*}
		\mathcal{K}_1(z,a_t) &=  \eta^\circ(a_t)\Delta(t) \mathcal{K}(z, z  a_{-t})\\
		& = \eta^\circ(a_t) \Delta(t)  \int_\R  \sigma(z,\lambda ) \phi_{\lambda}(a_t) |c(\lambda)|^{-2} d\lambda.
	\end{align*}

	We will prove Lemma \ref{lemma:Coifman-Weiss_condition_our_case} in the following steps:\\
	
	\textbf{Step I : Separation of the Kernel $ \mathcal{K}_1 $}. 
	
	 Suppose $ \varPhi $ is a smooth even function on $ \R $ with $ 0\leq \varPhi\leq 1 $; $ \varPhi=1 $ when $ |\lambda|>2$; $ \varPhi(\lambda)=0 $ when $ |\lambda|<1 $. Then we claim 	
	\begin{equation}\label{eqn:separation_of_kernel}
			\mathcal{K}_1(z a_s, a_t)    =   \mathcal{K}_0(za_s,t) + \zeta(za_s,t),
	\end{equation} 
 \text{ satisfying the followings: }\begin{enumerate}
 	\item[(i)]There is some constant $ C>0 $, such that
 	 \begin{equation}\label{eqn:condn_zeta_local_part}
 		 |\zeta(z, t)| \leq C\, \zeta_0(t) \in L^1(\R),\text{ for all   $ z \in G$} ,
 	\end{equation} \text{and}\\
 	\item [(ii)] $\begin{aligned}[t]
 		 \mathcal{K}_0(za_s,t)  =  C \eta^\circ(a_t )   \left( \frac{t^{d-1}}{\Delta(t)}\right)^{1/2} \int_{0}^\infty \Phi(\lambda)  \sigma(z a_s,\lambda ) \mathcal{J}_{(d-2)/2} (\lambda t) |c(\lambda)|^{-2} d\lambda,
 	\end{aligned}$
 \end{enumerate}

The proof of the above claim will be based on the local expansion of $ \phi_{\lambda} $ from Theorem  \ref{thm:localexpansion-phi} and idea of \cite{Stanton and Tomas}. First  we extend the result \cite[Proposition 4.1]{Stanton and Tomas} to a more general setting.
\begin{proposition}\label{prop:expression_of_p(z,lambda)}
	Let $ p  : G\times \R^+ \rightarrow \C $ be  a smooth  function satisfying the estimates \begin{equation}
		\begin{aligned}
			\partial_\lambda^\alpha p(z, 0) =0  \quad \text{ when $  0\leq \alpha \leq \left[\frac{d+1}{2}\right]+1  $},\\
			\left| \partial_\lambda^\alpha p(z, \lambda) \right| \leq C_\alpha (1+|\lambda |)^{-\alpha} \quad 0\leq \alpha \leq  \left[\frac{d+1}{2}\right]+1,
		\end{aligned}
	\end{equation}
	for all $ z \in G $ and the constants $ C_\alpha $'s  are independent of $ z\in G. $ Then there exists  functions $ e_z, e_0 \in L^1({G//K})$ such that 
	\begin{equation}\label{eqn:inversion_of_p(z,lambda)}
		\begin{aligned}
			\check p(z, a_t) \eta^\circ(a_t)    &:= \eta^\circ (a_t)  \int_{0}^{\infty} p(z,\lambda) \phi_{\lambda}(a_t) |c(\lambda)|^{-2} d\lambda \\
			&  =  C \eta^\circ (a_t)  \left(\frac{t^{d-1}}{\Delta(t)}\right)^{\frac{1}{2}} \int_{0}^\infty p(z,\lambda) \mathcal{J}_{(d-2)/2}(\lambda t )   |c(\lambda)|^{-2} d\lambda +e_z(t),
		\end{aligned}
	\end{equation}
	and \begin{equation}\label{eqn:est_of_e_z(t)}
		   | e_z(t) | \leq C \, e_0(t) \text{ for all } z\in G .
	\end{equation}
\end{proposition}

\begin{proof}[Proof of Proposition \ref{prop:expression_of_p(z,lambda)} ]
	 Applying Theorem \ref{thm:localexpansion-phi}, with $ M $  chosen to be $ N>(d+1)/2 $ and defining
	 \begin{equation}
	 	\begin{aligned}
	 		 e_z(t) =& \, C \eta^\circ (a_t)  \left(\frac{t^{d-1}}{\Delta(t)}\right)^{\frac{1}{2}} \sum_{m=0}^{N}  t^{2m} a_m(t)\int_{0}^\infty p(z,\lambda) \mathcal{J}_{(d-2)/2}(\lambda t )   |c(\lambda)|^{-2} d\lambda\\
	 		 &+  \eta^\circ (a_t)   \int_{0}^\infty p(z,\lambda)  E_{N+1}(\lambda t)   |c(\lambda)|^{-2} d\lambda,
	 	\end{aligned}
	 \end{equation}
 we get \eqref{eqn:inversion_of_p(z,lambda)}. Now one can complete the proof following the technique of \cite[Proposition 4.1]{Clerc Stein}. Indeed one   needs to utilize    the fact that the derivatives of $ p(z,\lambda) $ have bounds  independent of $ z $ and this will  eventually     lead us to  \eqref{eqn:est_of_e_z(t)}.
	\end{proof}
 
Now coming back to the proof of our claim  \eqref{eqn:separation_of_kernel},    we can write
 \begin{align*}
	\mathcal{K}_1(za_s,a_t) &=  \eta^\circ(a_t) \Delta(t) \int_\R \varPhi(\lambda) \sigma(za_s,\lambda ) \phi_{\lambda}(a_t) |c(\lambda)|^{-2} d\lambda\\
	&	\quad \quad + \eta^\circ(a_t)  \Delta(t) \int_\R  (1-\varPhi(\lambda))\sigma(za_s,\lambda ) \phi_{\lambda}(a_t) |c(\lambda)|^{-2} d\lambda.
\end{align*}
We observe that $ \varPhi(\lambda) \sigma(z a_s,\lambda) $ satisfy the hypotheses of  Proposition \ref{prop:expression_of_p(z,lambda)}; so we choose
\begin{align*}
	\zeta(z a_s,t) =\Delta(t) e_{z a_s}(t) + \eta^\circ(a_t) \int_\R  (1-\Phi(\lambda))\sigma(za_s,\lambda ) \phi_{\lambda}(a_t) |c(\lambda)|^{-2} d\lambda.
\end{align*}
Since $ |e_{z a_s}(t) |\leq C e_0(t) $ and $e_0\in L^1(G//K)$, so $ \Delta(t) e_{z a_s}(t) \in L^1(\R)$.
The second  term of $\zeta(z a_s,t) $ is bounded by $ \eta^\circ \cdot \int_{0}^2 |\sigma(za_s,\lambda )|c(\lambda)|^{-2} d\lambda  \leq C \, \eta^\circ$, ($ C  $ is independent of $ z $). Therefore the function $ \zeta $ satisfy the inequality \eqref{eqn:condn_zeta_local_part}, this also concludes our first step \eqref{eqn:separation_of_kernel}. \\

\textbf{Step II : Connection with Euclidean pseudo differential operator.}

Following \textbf{ Step I}, it is clear that to establish Lemma \ref{lemma:Coifman-Weiss_condition_our_case}; we need to prove $\mathcal{K}_0$ satisfies  \eqref{eqn:Coifman-Weiss_condn_satisified_by_K}.  We will do that by showing $ \{\mathcal{K}_0(z a_s, t) : z\in G \}$ are kernels of Euclidean pseudo differential operators corresponding to a family of symbols $ \{ a_z(s,y): z\in G \} $, belonging  to the symbol class $ \mathcal{S}^0. $ In fact, we will prove the following  

 \begin{equation}\label{eqn:condn_on_K_0}
	\left| \partial_s^\beta \partial_y^\alpha \int_{-\infty}^{\infty} e^{-2\pi ixy } \mathcal{K}_0(za_s,x)dx \right| \leq C_{\alpha,\beta } (1+|y|)^{-\alpha}  \quad \alpha, \beta =0,1,2,
\end{equation}
for all $z\in G$.   Assuming the inequality above, we will  complete the proof of  Lemma \ref{lemma:Coifman-Weiss_condition_our_case} below.

\noindent
\begin{proof}[\textbf{Proof of  Lemma \ref{lemma:Coifman-Weiss_condition_our_case} assuming \eqref{eqn:condn_on_K_0} : }]
	
For each $z\in G$, we define
\bes
a_z(s,y):=   \int_{-\infty}^{\infty} e^{-2\pi ixy } \mathcal{K}_0(za_s,x) dx, s, y\in \R.
\ees
We  observe that the constants $ C_{\alpha,\beta} $ in \eqref{eqn:condn_on_K_0} are independent of $ z\in G $. Hence from (\ref{eqn:condn_on_K_0}) it follows that the family of symbols $ \{a_z(s,y) : z \in G\} $ satisfies the hypothesis of Theorem \ref{thm:pdo_for_family_of_symbols} and hence (\ref{eqn:Coifman-Weiss_condn_satisified_by_K}) holds, which also concludes the Lemma  \ref{lemma:Coifman-Weiss_condition_our_case}.
\end{proof}

Now it only remains to prove that $\mathcal{K}_0$ satisfies  \eqref{eqn:condn_on_K_0}.  We recall the definition of $ \mathcal{K}_0 $ from \eqref{eqn:separation_of_kernel},
\begin{equation}
	 \mathcal{K}_0(za_s,t)  =  C \eta^\circ(a_t )   \left( \frac{t^{d-1}}{\Delta(t)}\right)^{1/2} \int_{0}^\infty \Phi(\lambda)  \sigma(z a_s,\lambda ) \mathcal{J}_{(d-2)/2} (\lambda t) |c(\lambda)|^{-2} d\lambda,
\end{equation}
 Next we will consider  separately the cases $ d $ odd and $ d $ even.

When $ d $ is even, we write
\begin{equation*}
	\mathcal{J}_{(d-2)/2}= C(z^{-1} \partial_z)^{(d-2)/2} \mathcal{J}_0(z),
\end{equation*}
and 
\begin{align*}
	\mathcal{J}_0 (\lambda t) =\frac{2}{\pi} \int_{\lambda}^{\infty} (\mu^2-\lambda^2)^{-1/2} \sin \mu t d\mu.
\end{align*}
Let $ q(z,\lambda) :=  (\partial_\lambda \cdot (1/\lambda)) ^{(d-2)/2} \left(  \Phi(\lambda) \sigma(z,\lambda) |c(\lambda)|^{-2}     \right)$. Then using integration by parts we have,
\begin{align*}
	\mathcal{K}_0(z a_s, t) &=C \eta^\circ(a_t) \left[ \frac{\Delta(t)}{t^{d-1}} \right]^{\frac{1}{2}}\, t \int_{\R} q(za_s,\lambda)  \mathcal{J}_0 (\lambda t) d\lambda \\
	& = C \eta^\circ(a_t) \left[ \frac{\Delta(t)}{t^{d-1}} \right]^{\frac{1}{2}}\, t \int_{\R}   \sin \mu t \int_0^\mu  q(za_s,\lambda)  (\mu^2-\lambda^2)^{-1/2} d\lambda \,d\mu \\
	& =  C \eta^\circ(a_t) \left[ \frac{\Delta(t)}{t^{d-1}} \right]^{\frac{1}{2}}\,  \int_{\R}   \cos \mu t  \frac{d}{d \mu }\int_0^\mu  q(za_s,\lambda)  (\mu^2-\lambda^2)^{-1/2} d\lambda\, d\mu.
\end{align*}

Let us define
\begin{equation}\label{eqn:defn_of_g(z,mu)}
	\begin{aligned}
		g(z,\mu) &= \int_{0}^\mu (\mu^2-\lambda^2)^{-1/2}  q(z,\lambda) \,d\lambda,\quad z\in G,  \mu\geq 0\\
		&= \frac{1}{2}   \int_{-1}^{1}  (1-\lambda^2)^{-1/2} q(z,\lambda \mu)\, d\lambda  \quad z\in G,  \mu>0.
	\end{aligned}
\end{equation}
and 
\begin{eqnarray*}
	h(z,\mu) = \left( \frac{\partial}{\partial\mu} g \right) (z,|\mu|) , \quad \mu \in \R.
\end{eqnarray*}
Then   we obtain
\begin{align*}
	\mathcal{K}_0(za_s,t)&=  -C \eta^\circ(a_t) \left[ \frac{\Delta(t)}{t^{d-1}} \right]^{1/2}\,  \int_0^\infty     \frac{d}{d\mu} \, g (za_s,\mu) \cos \mu t  \, d\mu\\
	&=   -C \eta^\circ(a_t)\left[ \frac{\Delta(t)}{t^{d-1}} \right]^{1/2}\,  \int_{-\infty}^\infty   h(za_s,\mu) e^{i\mu t} d\mu.
\end{align*}
Let $ \nu(t) =-{C \eta^\circ(a_t)} \left[ \frac{\Delta(t)}{t^{d-1}} \right]^{1/2}, \,  t \in \R$. Then \be\label{eqn:ftK0}
\mathcal{F}(\mathcal{K}_0(za_s,\cdot) )(\mu) =\left( \mathcal{F}(\nu) \ast_{\R} h(za_s,\cdot)\right)(\mu).
\ee
First we claim that  \begin{equation}\label{eqn:est_required_for_h}
	\sup_{\mu \in \R} \left( \left| \partial_s^\beta h(za_s, \mu)\right| +\left|  (1+\mu)\partial_s^\beta \partial_\mu h(za_s, \mu)\right|\right) <C_\beta\quad \text{ for all } \beta,
\end{equation}
where the constants $ C_\beta $ are independent of $ z\in G $.  From  \eqref{eqn:defn_of_g(z,mu)} we have
\begin{align*}
	h(za_s, \mu) = \frac{1}{2}   \int_{-1}^{1}  (1-\lambda^2)^{-1/2} \lambda \frac{\partial}{\partial (\lambda \mu)}q(z,\lambda \mu)\, d\lambda \quad \text{ for $ \mu >0 $}.
\end{align*} 
Then using (\ref{eqn:hypothesis_of_sigma})  and ($\ref{lemma:est_of_|c(-lambda)|^-2}$) we get
\begin{equation}
	\left| \frac{\partial^\beta}{\partial s^\beta}\frac{\partial^l}{\partial \lambda^l} q(za_s,\lambda)\right|  \leq C_{\beta,l}(1+|\lambda|)^{1-l} \quad  \text{for all } \beta, l.
\end{equation}
Hence \begin{align*}
	|\frac{\partial^\beta}{\partial s^\beta} h(za_s,\mu)| &\leq C_\beta \int_{0}^1 (1-\lambda^2)^{-1/2} \lambda \,d\lambda\\
	& \leq C_\beta,
\end{align*}
Also,  \begin{align*}
	|(1+\mu)\frac{\partial^\beta}{\partial s^\beta}  \frac{\partial}{\partial \mu}  h(za_s,\mu)| &\leq C_\beta \int_{0}^1 (1-\lambda^2)^{-1/2} \frac{\lambda^2}{(1+|\lambda \mu| )}  d\lambda + C_\beta |\mu| \int_{0}^1 (1-\lambda^2)^{-1/2} \frac{\lambda^2}{(1+|\lambda \mu| )}  d\lambda \\
	&\leq C_\beta \int_{0}^1 (1-\lambda^2)^{-1/2}  {\lambda^2}  d\lambda+ C_\beta |\mu| \int_{0}^1 (1-\lambda^2)^{-1/2} \frac{\lambda}{| \mu| }  d\lambda \\
	&\leq C_\beta.
\end{align*}
Similarly we can prove that 
\bes
\left|  (1+\mu)^\alpha \partial_s^\beta \partial_\mu^\alpha h(za_s, \mu)\right|\leq C_{\alpha, \beta} \, \text{ for all } \alpha, \beta.
\ees
Therefore, from (\ref{eqn:ftK0}), we conclude that $\mathcal{K}_0$ satisfies (\ref{eqn:condn_on_K_0}) using the fact that $ \mathcal{F}(\nu) $ is a Schwartz function.

When $ d $ is odd, we  write 
\bes \mathcal{J}_{(d-2)/2} (z) =C (z^{-1} \partial_z)^{(d-1)/2} (\cos z).\ees
After integration by parts, we see that it suffices to prove  $  ( \partial_\lambda\cdot (1/\lambda))^{(d-1)/2} (\Phi(\lambda)\sigma(za_s,t) |c(\lambda)|^{-2})$  satisfies \eqref{eqn:condn_on_K_0}, which follows from the estimates \eqref{estimate of |c(-lambda)|^-2} on $ |c(\lambda)|^{-2} $ and the hypothesis on $ \sigma $.


\section{\textbf{Global analysis of the pseudo differential operator $ \Psi_{\sigma} $}}\label{analysis on global part}

In the previous section, we saw how $ \mathcal{K}_0 $ behaved like a kernel of Euclidean pseudo differential operators, and using the transference method we got a  bound for  $ L^p $ operator norm of  the local part of  $ \Psi_{\sigma}$. But the analysis on the global part changes drastically as the local and global behavior of $ \phi_{\lambda} $   entirely different.   
In multiplier theory the author in \cite{Ionescu 2000} used a transference theorem for convolution operator which is comparable to the Herz majorizing principle. For higher rank  case  the author used the same principle  to estimate the $ L^p $ norm of multiplier operator (see \cite[Lemma 4.3]{Ionescu 2002}). Clearly such tools are not applicable in our setting. 

In this section we will prove the $ L^p $ boundedness of $ \Psi_{\sigma}^{glo} $. Here we will use the expansion of  spherical function away from the origin (Proposition \ref{propn:global_expansion_of_phi_lambda}), and will see the global analysis of $ \Psi_{\sigma} $ has no Euclidean analogue.   The broad strokes of our approach in this section follow that of \cite{Ionescu 2000}. Also Theorem \ref{thm:pdo_for_sym_space} will be derived as a consequence of following theorem and Theorem \ref{thm:Lp_boundedness_of_T_sigma_local}. Before stating the main result in this section, let us recall the definition of  $ \Psi^{ {glo}}_{\sigma}$, the global part of the operator $ \Psi_\sigma $ from \eqref{eqn:defn_of_T_sigma_global},
 \begin{equation}
	  \Psi^{ {glo}}_{\sigma}f(x) =\int_{\X} f(y)\eta(y^{-1}x) \mathcal{K}(x,y) \, dy ,
\end{equation}
where 
\bes
 \mathcal{K}(x,y) = \int_{\R} \sigma (x,\lambda) \phi_\lambda(y^{-1}x ) |c(\lambda)|^{-2} d\lambda.
\ees

 Then we have the following:
 \begin{theorem}\label{thm:Lp_boundedness_of_T_sigma_global}
Suppose  $ p\in(1,2)\cup(2,\infty) $. Then there is a constant $C>0$ such that \be \label{eqn:est_of_Lp_norm_of_Tsigma^glo_f}\|\Psi^{glo}_{\sigma}f\|_{L^p(\X)}\leq C \|f\|_{L^p(\X)},
\ee
 for all $ f\in L^p(\X) $. 
 \end{theorem}


Suppose $ \chi^+ $ and $ \chi^- $ be the characteristic functions of  the intervals $ [0,\infty) $ and $ (-\infty, 0) $ respectively.  {{Next we define the following operators on the $ \ol N A $ group,}}
\begin{equation}\label{eqn_defn_of_E^-}
	\begin{aligned}
	 \mathfrak{E}^-{\mathfrak f}(\ol n a_t) &=  \int_{\ol N} \int_{\R} \mathfrak f(\ol ma_s )  \eta(\delta_{-s}({\ol m}^{-1}\ol n ) a_{t-s}) \bigg(\int_{ \R}\sigma(\ol  n a_t, \lambda)   \phi_\lambda(\delta_{-s}({\ol m}^{-1}\ol n ) a_{t-s})  |c(\lambda)|^{-2}\, d\lambda\bigg)\\
	 &\cdot  \chi^-(t-s)\,e^{2\rho s}  |c(\lambda)|^{-2}\, d\lambda\,  ds\,     d{\ol  m},
\end{aligned}
\end{equation}
and 
\begin{equation}\label{eqn_defn_of_E^+}
	\begin{aligned}
	 \mathfrak{E}^+\mathfrak f(\ol n a_t) &=  \int_{\ol N} \int_{\R} \mathfrak f(\ol ma_s )  \eta(\delta_{-s}({\ol m}^{-1}\ol n ) a_{t-s}) \bigg(\int_{ \R}\sigma(\ol  n a_t, \lambda)   \phi_\lambda(\delta_{-s}({\ol m}^{-1}\ol n ) a_{t-s})  |c(\lambda)|^{-2}\, d\lambda\bigg)\\
	 &\cdot  \chi^+(t-s)\,e^{2\rho s}  |c(\lambda)|^{-2}\, d\lambda\,  ds\,     d{\ol  m}.
\end{aligned}
\end{equation} 
Equipped with the expression of  above, we aim to prove the  $ L^p $ norm estimates of $ \mathfrak{E}^{\pm} $ in  $ {\ol N}A$ group. The reason for this is that,  
 \begin{equation}
	\Psi_{\sigma}^{glo}\mathfrak f(\cdot)= \mathfrak{E}^-\mathfrak f(\cdot) +  \mathfrak{E}^+ \mathfrak f (\cdot), \quad \text{ in the $ \ol N A $ group},
\end{equation}
and so the $ L^p(\ol N A) $ operator norm of  $\mathfrak{E}^{-} +\mathfrak{E}^{+}   $ dominates the operator norm of $ \Psi_{\sigma}^{glo} $,  { $||| \Psi_{\sigma}^{glo}|||_{L^p(\ol N A)} $, which in turn equals to   $||| \Psi_{\sigma}^{glo}|||_{L^p(\X)} $.} Thus the required estimate  of $ \Psi_{\sigma}^{glo} $ of desired form \eqref{eqn:est_of_Lp_norm_of_Tsigma^glo_f} will follow, if we show that $ \mathfrak{E}^{\pm} $ maps $ L^{p}(\ol N A) $ boundedly to itself with operator norm bounded by constant. So in the reminder of this section we focus on the boundedness of $ \mathfrak{E}^{\pm} $.

Before proceeding  further we recall  that an we may assume that  $ \sigma(\cdot, \lambda)$ has exponential  decay as $ \lambda \rightarrow \infty $. In fact, we assume  $ \sigma(\cdot, \lambda) $ is multiplied with a factor of the form  $ e^{-\epsilon\lambda^2} $, where $  0 <\epsilon\leq1 $, and then sending $ \epsilon  $ towards zero, we will get the required estimates. 
Let us summarize this discussion by writing: suppose $\mathfrak{f}, \mathfrak{h}$ be two compactly supported functions on the group $ \ol N A $, and let us denote
\begin{equation}\label{eqn:defn_of_E^pm}
	\begin{aligned}
		\left\langle\mathfrak{E}^{\pm}_\epsilon \mathfrak{f}, \mathfrak{h} \right\rangle   :=& \int_{\ol N} \int_{\ol N} \int_{\R }\int_{ \R}\int_{\R}  \mathfrak{f}(\overline{m}a_s)\mathfrak{h}(\overline{n}a_t)  \eta (\delta_{-s}(\overline{m}^{-1}\overline{n})a_{t-s})   \sigma(\overline{n}a_t,\lambda) \phi_\lambda(\delta_{-s}(\overline{m}^{-1}\overline{n})a_{t-s})   \\ 
		& \cdot   \chi^{\pm}(t-s)  e^{2\rho(t+s)} e^{-\epsilon \lambda^2}|c(\lambda)|^{-2}  \,  d\lambda\,  dt\,  ds\, d\ol{n} \,d\ol m.
	\end{aligned}
\end{equation}
Then we aim to prove 
\begin{equation}\label{eqn:req_est_E^pm_e}
	 \lim \limits_{\epsilon \rightarrow 0} \left|\left\langle \mathfrak{E}^{\pm}_{\epsilon} \mathfrak{f,h}\right\rangle \right| \leq C \|\mathfrak{f}\|_{L^p(\ol N A)} \|\mathfrak h\|_{L^{p'}(\ol N A)},
\end{equation}
for some constant $ C>0 $ and $ p'=p/(p-1 )$ is the conjugate exponent of $ p $. Now we proceed to prove \eqref{eqn:req_est_E^pm_e}. By a change of variable \eqref{eqn:defn_of_E^pm} transforms to,
\begin{equation} \label{eqn:equiv_defn_E^pm}
	\begin{aligned}
		\left\langle\mathfrak{E}^{\pm}_\epsilon \mathfrak{f}, \mathfrak{h} \right\rangle   =& \int_{\ol N} \int_{\ol N} \int_{\R }\int_{ \R}\int_{\R}  \mathfrak{f}(\overline{m}a_s)\mathfrak{h}(\ol {m}\,\overline{v}a_t) \eta(\delta_{-s}(\overline{v})a_{t-s})         \sigma(\ol {m}\,\overline{v}a_t,\lambda) \phi_\lambda(\delta_{-s}(\overline{v})a_{t-s})    \\
		& \cdot \chi^{\pm}(t-s) e^{2\rho(t+s)}\,   e^{-\epsilon \lambda^2}|c(\lambda)|^{-2}d\lambda\,  dt \, ds\, d\ol{v} \,d\ol m.
	\end{aligned}
\end{equation}
\subsection{\textbf{ $ L^p$ operator norm estimate of  $ \mathfrak{E} ^-$}}  \hfil

First we prove the estimate \eqref{eqn:req_est_E^pm_e} for $ \mathfrak{E}^- $. Here we will see how the holomorphic extension property of $ \sigma(\cdot,\lambda) $ corresponds to an exponential decay. By writing $  t=s-r$ and using \eqref{eqn:dialation_formula_for_N-}, we get from \eqref{eqn:equiv_defn_E^pm}
\begin{equation}\label{eqn:final_expression_of_E^-}
\begin{aligned}
		\left\langle\mathfrak{E}^{-}_\epsilon \mathfrak{f}, \mathfrak{h} \right\rangle = &\int_{r=0}^\infty  \int_{\ol N} \int_{\ol N} \int_{\R }\int_{ \R} \mathfrak f(\overline{m}a_s)\mathfrak h(\overline{m} \, \delta_s(\ol w) \,a_{s-r})  \eta(\ol w \,a_{-r})  \sigma(\overline{m} \, \delta_s(\ol w)\, a_{s-r},\lambda)     \\
	&\cdot  \,\phi_\lambda(\ol w \,a_{-r})   e^{2\rho(s-r)}\,  e^{-\epsilon \lambda^2}|c(\lambda)|^{-2} d\lambda  \, ds\, d\ol{w} \,d\ol m\,  dr.
\end{aligned}
\end{equation}
Now we will use Harish-chandra expansion of the spherical function $ \phi_{\lambda}$. So we substitute the expansion \eqref{eqn:phi-lambda-assymp-expansion} into \eqref{eqn:final_expression_of_E^-} and notice the integrand is $ W $-invariant. One has 

\begin{equation}\label{eqn:I^{1,epsilon}=I^1_{1}+I^1_{2}}
	\begin{aligned}
		\left\langle\mathfrak{E}^{-}_\epsilon \mathfrak{f}, \mathfrak{h} \right\rangle =&2  \int_{r=0}^\infty \int_{\ol N} \int_{\ol N} \int_{\R }\int_{ \R} \mathfrak f(\overline{m}a_s)\mathfrak h(\overline{m} \, \delta_s(\ol w) a_{s-r})   \eta(\ol w \,a_{-r})  \sigma(\overline{m} \, \delta_s(\ol w) a_{s-r},\lambda)\\
		&\cdot    e^{2\rho(s-r)}  \,e^{(i\lambda-\rho)[\ol w a_{-r}]^+}   \frac{e^{-\epsilon \lambda^2}}{c(-\lambda)} d\lambda  \, ds\, d\ol{w} \,d\ol m\,  dr \\
		+&2 \int_{r=0}^\infty  \int_{\ol N} \int_{\ol N} \int_{\R }\int_{ \R} \mathfrak f(\overline{m}a_s)\mathfrak h(\overline{m} \, \delta_s(\ol w) \,a_{s-r})    \eta(\ol w \,a_{-r}) \sigma(\overline{m} \, \delta_s(\ol w) a_{s-r},\lambda)   \\ 
		&\cdot   e^{2\rho(s-r)}\,e^{(i\lambda-\rho-2)[\ol w a_{-r}]^+}a(\lambda,[\ol w a_{-r}]^+)  \frac{e^{-\epsilon \lambda^2}}{c(-\lambda)} 
		\, d\lambda  \, ds\, d\ol{w} \,d\ol m\,  dr\\
		&=	\left\langle\mathfrak{E}^{-}_{0,\epsilon} \mathfrak{f}, \mathfrak{h} \right\rangle +	\left\langle\mathfrak{E}^{-}_{e,\epsilon} \mathfrak{f}, \mathfrak{h} \right\rangle.
	\end{aligned}
\end{equation}
We recall the function $ \lambda\rightarrow a(\lambda,\cdot) $ satisfy   favorable symbol-type estimate (only for rank one case). Moreover, by Leibniz rule one has    for all $ j\in [0,N] $
		\begin{equation}\label{eqn:est_of_sigma*a*c^{-1}}
		\lim_{\epsilon \rightarrow 0} \left| \frac{\partial^j}{\partial \lambda^j} \left( \frac{\sigma(\ol n a,\lambda)\, a(\lambda,r)e^{-\epsilon \lambda^2}}{c(-\lambda)}\right)\right| \leq \frac{C_j}{(1+|\lambda|)^{j-\frac{(d-1)}{2}}}  
	\end{equation}
for all $ \lambda \in \C$  with $ 0 \leq \Im \lambda \leq \rho_p $, $ r\geq1 $ and $ \ol n a \in \ol N A  $. Therefore the required estimate of $ \left\langle\mathfrak{E}^{-}_{e,\epsilon} \mathfrak{f}, \mathfrak{h} \right\rangle $, induced by the error term $ a(\lambda,[\ol w a_{-r}]^+) $, follows similarly as of  $ \left\langle\mathfrak{E}^{-}_{0,\epsilon} \mathfrak{f}, \mathfrak{h} \right\rangle $. So we only need to prove the lemma below.

 \begin{lemma}\label{lemma:est_of_I^{1,epsilon}_{1}}
	We have  \begin{equation}\label{eqn:est_of_E_0^-}
		\lim_{\epsilon\rightarrow 0}\left| \left\langle\mathfrak{E}^{-}_{0,\epsilon} \mathfrak{f}, \mathfrak{h} \right\rangle\right|	 \leq C \| \mathfrak f\|_{L^p(\ol N A )} \,  \| \mathfrak{ h} \|_{L^{p'}(\ol N A)}.
	\end{equation}
\end{lemma}
\begin{remark}\begin{enumerate}
		\item 	 We observe that, we have not used the holomorphic property of the symbol yet. In this lemma we will start using the holomorphic condition, and it will be clear how it   corresponds to an exponential decay.
\item   {	We will need the explicit expression of $ [\ol w a_{r}]^+ $, which is only available for $ r\geq 0 $, so we will use the evenness of Abel transform, whenever is $ r $ is negative.} 
\end{enumerate}
\end{remark}

\begin{proof}[Proof of Lemma \ref{lemma:est_of_I^{1,epsilon}_{1}}] 
	 
	 We begin by moving  the integration with respect to $ \lambda $ in \eqref{eqn:I^{1,epsilon}=I^1_{1}+I^1_{2}}  from $ \R $ to the line $ \R+ i (\rho_p-\rho/ [\ol w a_{-r}]^+) $, to get
	 
	 \begin{align*}
	 	\left\langle\mathfrak{E}^{-}_{0,\epsilon} \mathfrak{f}, \mathfrak{h} \right\rangle=  & C \int_{r=0}^\infty  \int_{\ol N} \int_{\ol N} \int_{\R }  \mathfrak f(\overline{m}a_s) \mathfrak h (\overline{m}  \, \delta_s(\ol w) \,a_{s-r}) \eta(\ol w a_{{-r}}) \\ 
	 	& \hspace*{-1.9cm}\cdot   \int_\R \vartheta_\epsilon \left(\overline{m} \, \delta_s(\ol w)a_{s-r},\lambda+i(\rho_p-\frac{\rho}{[\ol w a_{-r}]^+})\right) \,    e^{\left(i\lambda-\frac{2\rho}{p}+\frac{\rho}{[\ol w a_{-r}]^+}\right)[\ol w a_{-r}]^+}  d\lambda \,
	 	e^{2\rho(s-r)}\, ds\, d\ol{w} \, d\ol m\,  dr,
	 \end{align*}
 where 
 \be\label{eqn:sigma_epsilon} \vartheta_\epsilon(x,\lambda):=\dfrac{\sigma(x,\lambda)}{c(-\lambda)} e^{-\epsilon \lambda^2}.
 \ee 
 In order to effectively harness the oscillation of $ e^{i\lambda [\ol w a_{-r}]^+} $, we use integration by parts $ l\,(>\frac{d+1}{2}) $ times to obtain, 
	 
	 \begin{equation}\label{eqn:int_byparts_of_sigma_{epsilon}*e^{i lambda}}
	 	\begin{aligned}
	 		\int_\R \vartheta_\epsilon \left(\overline{m} \, \delta_s(\ol w)a_{s-r},\lambda+i(\rho_p-\frac{\rho}{[\ol w a_{-r}]^+})\right) \,e^{i\lambda[\ol w a_{-r}]^+}    d\lambda \\=  C \int_\R \frac{\partial^{l}}{\partial\lambda^{l}}\vartheta_\epsilon\left(\overline{m} \, \delta_s(\ol w)a_{s-r},\lambda+i(\rho_p-\frac{\rho}{[\ol w a_{-r}]^+})\right)
	 		\cdot \frac{e^{i\lambda[\ol w a_{-r}]^+}}{([\ol w a_{-r}]^+)^{l}}  d\lambda.
	 	\end{aligned}
	 \end{equation}
	 Now we define 
	 \begin{equation}\label{eqn:def_of_F_and_H}
	 	\begin{aligned}
	 		\mathfrak{F}(s) =& \left[ \int_{\ol N} |\mathfrak f(\ol m a_s)|^p d\ol m \right]^{\frac{1}{p}},\\
	 		\mathfrak{H}(s)= & \left[ \int_{\ol N} |\mathfrak{ h}(\ol m a_s)|^{p'} d\ol m \right]^{\frac{1}{p'}}.
	 	\end{aligned}
	 \end{equation}
After applying the H\"older's inequality, substituting \eqref{eqn:def_of_F_and_H}  gives  us the absolute value  of $	\left\langle\mathfrak{E}^{-}_{0,\epsilon} \mathfrak{f}, \mathfrak{h} \right\rangle$ dominated by

\begin{equation}
	\begin{aligned}
		&C \int_{r=0}^{\infty} \int_{\R}\mathfrak F (s) \mathfrak H(s-r) \int_{\ol N} \eta (\ol w a_{-r})\frac{ e^{(-\frac{2}{p}+\frac{1}{ [\ol w a_{-r}]^+})\rho[\ol w a_{-r}]^+}}{([\ol w a_{-r}]^+)^{l}} \,  \\
		&\hspace*{2cm}\cdot  \int_\R \|\frac{\partial^{l}}{\partial\lambda^{l}}\vartheta_\epsilon\left((\cdot),\lambda+i(\rho_p-\frac{\rho}{[\ol w a_{-r}]^+})\right)\|_{\infty} \, d\lambda\, d \ol w \, e^{2\rho(s-r)}\,  ds \,dr\\
		=&C \int_{r=0}^{\infty} \int_{\R}\mathfrak F(s) \mathfrak H(s-r) \int_{\ol N}  \eta (\ol w a_{r}) \frac{ e^{(-\frac{2}{p}+\frac{1}{ [\ol w a_r]^+})\rho[\ol w a_{r}]^+}}{([\ol w a_{r}]^+)^{l}} \,   \\
		&\hspace*{2cm}\cdot  \int_\R \|\frac{\partial^{l}}{\partial\lambda^{l}}\vartheta_\epsilon\left((\cdot),\lambda+i(\rho_p-\frac{\rho}{[\ol w a_{r}]^+})\right)\|_{\infty}\,  d\lambda \, d\ol w \,e^{2\rho s}\, ds \, dr,
	\end{aligned}
\end{equation}
where in the last line we invoked the evenness of Abel transform to change $ [w a_{-r}]^+ $ to  $[ w a_r]^+ $. We recall $ \eta $ is a  function supported in $ [1,\infty) $. Using this, Lemma \ref{lemma:estimate_of_sigma_(lambda)_(c(-lambda)_e(-epsilo^2_lamda))} and substituting the expression of $ [w a_r]^+ $ \eqref{eqn:expression_of_[va_r]}, yields  the following
\begin{equation*}
	\begin{aligned}
		&\lim_{\epsilon\rightarrow 0} \left| 	\left\langle\mathfrak{E}^{-}_{0,\epsilon} \mathfrak{f}, \mathfrak{h} \right\rangle \right| \\&\leq C   \int_{r=0}^{1} \left(\int_{\R}\mathfrak{F}(s) e^{\frac{2\rho}{p}s} \mathfrak{H}(s-r)e^{\frac{2\rho}{p'}(s-r)}  ds \right)\int_{\ol N}{ e^{-\frac{2}{p}\rho( H(\ol w) )}}\, d\ol w    \int_\R \frac{1}{(1+|\lambda|)^{l-{\frac{d-1}{2}}}}  d\lambda \,   dr\\
		&+ \int_{r=1}^{\infty} \left(\int_{\R}\mathfrak{F}(s) e^{\frac{2\rho}{p}s} \mathfrak{H}(s-r)e^{\frac{2\rho}{p'}(s-r)}  ds \right)\int_{\ol N}{ e^{-\frac{2}{p}\rho( H(\ol w) )}}\, d\ol w  \int_\R \frac{1}{(1+|\lambda|)^{l-{\frac{d-1}{2}}}}  d\lambda \,   \frac{e^{(\frac{2}{p'}-\frac{2}{p})\rho 
				r}}{{r^{l}} } dr.
	\end{aligned}
\end{equation*}
Finally we can reach \eqref{eqn:est_of_E_0^-} simply by an application of H\"older's inequality and \eqref{eqn:L1_norm_P(v-) (1+epsilon)_is_finite}.

\end{proof}


\subsection{\textbf{$ L^p $ operator norm estimate of  $ \mathfrak{E} ^+$}}\hfil

It remains to prove the operator norm estimate of $ \mathfrak{E}^+ $. Recalling the definition of $\left\langle\mathfrak{E}^{+}_{\epsilon} \mathfrak{f}, \mathfrak{h} \right\rangle$ from \eqref{eqn:equiv_defn_E^pm}, and after some changes of variables we have

\begin{equation} \label{eqn:equiv_defn_E^pm}
	\begin{aligned}
		\left\langle\mathfrak{E}^{+}_\epsilon \mathfrak{f}, \mathfrak{h} \right\rangle   =& \int_{\ol N} \int_{\ol N} \int_{\R }\int_{ \R}\int_{\R}  \mathfrak{f}(\overline{m}a_s)\mathfrak{h}(\ol {m}\,\overline{v}a_t) \eta(\delta_{-s}(\overline{v})a_{t-s})         \sigma(\ol {m}\,\overline{v}a_t,\lambda) \phi_\lambda(\delta_{-s}(\overline{v})a_{t-s})    \\
		& \cdot \chi^{+}(t-s) e^{2\rho(t+s)}\,   e^{-\epsilon \lambda^2}|c(\lambda)|^{-2}d\lambda\,  dt \, ds\, d\ol{v} \,d\ol m\\
		=& \int_{r=0}^\infty\int_{\ol  N}\int_{\ol N} \int_\R \int_\R \mathfrak f(\overline{m}a_s)\mathfrak h(\overline{m} \, \delta_s(\ol w) \,a_{s+r})  \eta(\ol w a_{r})\sigma(\overline{m}\, \delta_s{(\ol w)}\, a_{s+r},\lambda) \,\phi_\lambda(\ol w \,a_{r})       \\
		&\cdot   e^{2\rho(s+r)}\, e^{-\epsilon \lambda^2}|c(\lambda)|^{-2}d\lambda \, d\ol{w} \,d\ol m\,  ds\, dr.
	\end{aligned}
\end{equation}
We substitute Harish-chandra series expansion of $ \phi_{\lambda} $ \eqref{eqn:phi-lambda-assymp-expansion} into the expression above to get,
\begin{eqnarray}
	\left\langle\mathfrak{E}^{+}_\epsilon \mathfrak{f}, \mathfrak{h} \right\rangle   =  \left\langle\mathfrak{E}^{+}_{0,\epsilon} \mathfrak{f}, \mathfrak{h} \right\rangle   +\left\langle\mathfrak{E}^{+}_{e,\epsilon} \mathfrak{f}, \mathfrak{h} \right\rangle,   
\end{eqnarray}
where 
\begin{align}
	\begin{split}\label{eqn:final_defn_of_E_0^+}
 \left\langle\mathfrak{E}^{+}_{0,\epsilon} \mathfrak{f}, \mathfrak{h} \right\rangle =&2 \int_{r=0}^\infty \int_{\ol  N}\int_{\ol N} \int_\R \int_\R \mathfrak f(\overline{m}a_s)\mathfrak{h}(\overline{m} \, \delta_s(\ol w) \,a_{s+r}) \eta( \ol w a_r) \vartheta_{\epsilon}(\overline{m} \, \delta_s(\ol w)\, a_{s+r},\lambda) \,\\
 &\cdot e^{2\rho(s+r)}\, e^{(i\lambda-\rho)[\ol w a_{r}]^+}   \, d\lambda \, ds \,d\ol{w} \,d\ol  m\,  dr,
\end{split}\\
\begin{split}
 \left\langle\mathfrak{E}^{+}_{e,\epsilon} \mathfrak{f}, \mathfrak{h} \right\rangle,   =&2 \int_{r=0}^\infty\int_{\ol  N}\int_{\ol N} \int_\R \int_\R \mathfrak {f}(\overline{m}a_s) \mathfrak{h}(\overline{m} \, \delta_s(\ol w) \,a_{s+r}) \eta( \ol w a_r) \vartheta_{\epsilon}(\overline{m} \, \delta_s(\ol w)\, a_{s+r},\lambda) \,    \\
 & \cdot    e^{2\rho(s+r)}\, e^{(i\lambda-\rho-2)[\ol w a_{r}]^+}a(\lambda,[\ol w a_{r}]^+) \, d\lambda  \,ds \,d\ol{w} \,d\ol m\,  dr.
\end{split}
\end{align}
Now from \eqref{eqn:est_of_sigma*a*c^{-1}} and the discussion proceeding it, it's clear  \eqref{eqn:req_est_E^pm_e} will follow from the lemma below. The following lemma will also complete the proof of  Theorem \ref{thm:Lp_boundedness_of_T_sigma_global} for $ p\in (1,2) $.
\begin{lemma}\label{lemma:est_of_E_0^+}
	We have  \begin{equation}\label{eqn:est_of_E_0^+}
		\lim_{\epsilon\rightarrow 0}\left| \left\langle\mathfrak{E}^{+}_{0,\epsilon} \mathfrak{f}, \mathfrak{h} \right\rangle\right|	 \leq C \| \mathfrak f\|_{L^p(\ol N A )} \,  \| \mathfrak{ h} \|_{L^{p'}(\ol N A)}.
	\end{equation}
\end{lemma}
\begin{proof}
	 By moving  the integration with respect to $ \lambda $ in \eqref{eqn:final_defn_of_E_0^+}  from $ \R $ to the line $ \R+ i (\rho_p-\rho/ [\ol w a_{-r}]^+) $, it follows
	 \begin{equation*}\begin{split}
	 		\left\langle\mathfrak{E}^{+}_{0,\epsilon} \mathfrak{f}, \mathfrak{h} \right\rangle&= \int_{r=0}^\infty\int_{\ol  N}\int_{\ol N} \int_\R  \mathfrak{f}(\overline{m}a_s)\mathfrak{h}(\overline{m} \, \delta_s(\ol w) \,a_{s+r})  \eta( \ol w a_r)\\
	 		&\hspace*{-1.2cm}\cdot   \left(\int_{\R}\vartheta_{\epsilon}(\overline{ m} \, \delta_s (\ol w)\, a_{s+r},\lambda+i(\rho_p -\frac{\rho}{[\ol w a_{r}]^+})) \,e^{(i\lambda-\frac{2\rho}{p} +\frac{\rho}{[\ol w a_{r}]^+})[\ol w a_{r}]^+}   \, d\lambda\right) e^{2\rho(s+r)}  ds\,d\ol{m} \,d\ol w\,  dr\,  .
	 	\end{split}
	 \end{equation*}
 Now by using integration by parts and Lemma  \ref{lemma:estimate_of_sigma_(lambda)_(c(-lambda)_e(-epsilo^2_lamda))}, we get
 
 \begin{align*}
 	\lim\limits_{\epsilon \rightarrow 0}\left| \int_{\R}\vartheta_{\epsilon}(\overline{ m} \, \delta_s (\ol w)\, a_{s+r},\lambda+i (\rho_p-\frac{\rho}{[\ol w a_{r}]^+})) \,e^{i\lambda [\ol w a_{r}]^+}   \, d\lambda\right| & \leq \frac{C }{\left( [\ol w a_r]^+ \right)^{l}}\int_{\R} \frac{d\lambda}{(1+|\lambda|)^{l-\frac{d-1}{2}}}.
 \end{align*}
By taking $ l> \frac{d+1}{2} ,$  and substituting the expression $ [\ol w a_r]^+ $, it follows

\begin{equation}\nonumber\begin{aligned}
		\lim\limits_{\epsilon \rightarrow 0}	&\left| \left\langle\mathfrak{E}^{+}_{0,\epsilon} \mathfrak{f}, \mathfrak{h} \right\rangle\right|\\
		&\leq C \int_{r=0}^1 \int_{\R} \int_{\ol N}  \mathfrak{F}(s)\, e^{\frac{2\rho s}{p}}\, \mathfrak{H}(s+r)  \, e^{\frac{2\rho( s+r)}{p'}} \eta (\ol w a_r)\frac{e^{-\frac{2}{p}\rho H(\ol w)}  }{(r+ H(\ol w)+E(\ol w,r ))^{l}}d\ol w  \, ds\, dr\\
		&  + C \int_{r=1}^\infty \int_{\R} \int_{\ol N}  \mathfrak{F}(s)\, e^{\frac{2\rho s}{p}}\, \mathfrak{H}(s+r)  \, e^{\frac{2\rho( s+r)}{p'}} \eta(\ol w a_r) \frac{e^{-\frac{2}{p}\rho H(\ol w)}  }{(r+ H(\ol w)+E(\ol w,r ))^{l} }\, d\ol w  \, ds\, dr\\
		& \leq C  \| \mathfrak f\|_{L^p(\ol N A)} \,  \|\mathfrak  h\|_{L^{p'}(\ol N A)}. 
	\end{aligned}
\end{equation}
In the last inequality, we used H\"older's inequality along with \eqref{eqn:L1_norm_P(v-) (1+epsilon)_is_finite}. This completes the proof of the lemma and consequently  the proof of Theorem \ref{thm:pdo_for_sym_space}.

\end{proof}

\begin{remark}
	\begin{enumerate}
		\item  The proof of Lemma \ref{lemma:est_of_I^{1,epsilon}_{1}} and \ref{lemma:est_of_E_0^+}, for $ p>2 $ case is similar in spirit. Indeed one   needs to utilize  the fact that $ |(2/p-1) |=(1-2/p') $ for all $ p>2 $. Then those two lemmas (Lemma \ref{lemma:est_of_E_0^+} and \ref{lemma:est_of_I^{1,epsilon}_{1}}) will give us  Theorem \ref{thm:Lp_boundedness_of_T_sigma_global} and since Theorem \ref{thm:Lp_boundedness_of_T_sigma_local} is true for all $  p\in (1,\infty)$, we will finally have Theorem \ref{thm:pdo_for_sym_space} for all $p\in (1,2)\cup (2,\infty) $.
		\item 
		We observe from the proof of the global part, it follows that if we assume only
		\begin{equation}\label{eqn:hypothesis_of_sigma-global}
			\left| \frac{\partial^\alpha}{\partial \lambda ^\alpha} \sigma (x,\lambda)\right| \leq  {C_{\alpha}}{(1+|\lambda| )^{-\alpha}},
		\end{equation}
		for all $ \alpha=0,1,...\left[\frac{d+1}{2}\right]+1 $, $ x\in G$ and $ \lambda \in S_p $,
		in addition with holomorphicity condition on $S_p^\circ$, etc.		
		then the operator $ \Psi_\sigma^{glo} $ extends to a bounded operator on $ L^p{(\mathbb X)} $ to itself. That is, we do not need derivative condition on $x$ variable. The derivative condition on $x$ variable is required only to conclude the boundedness of the local part  $\Psi_\sigma^{loc}$ which follows from the boundedness of Euclidean pseudo differential operators. 
		\item  It is a natural question and perhaps an interesting problem to ask for $L^p$ boundedness of $\Psi_\sigma$  on $\X$ if $ \sigma $ is allowed to have singularities on boundary. We recall that on $\R^n$ this is not possible.
		\item Suppose $ \sigma:\X\times \C\rightarrow \C $ be a smooth function such that for each $x\in\X$, $\lambda\mapsto \sigma(x, \lambda)$ is an even, holomorphic function on $S_p^\circ$ and  satisfies \eqref{eqn:hypothesis_of_sigma} for some  $p\in (1, 2)$, that is \begin{equation}\label{eqn:hypothesis_of_sigma_on_L2}
			\left| \frac{\partial^\beta}{{\partial s}^\beta}\frac{\partial^\alpha}{\partial \lambda ^\alpha} \sigma (xa_s,\lambda)\right| \leq  {C_{\alpha,\beta}}{(1+|\lambda| )^{-\alpha}},
		\end{equation}
		for all $ \alpha=0,1,\cdots, \left[\frac{d+1}{2}\right]+1 ,\beta=0,1,2 $, $ x\in G, s\in \R$ and $ \lambda \in S_p $ for some  $p\in (1, 2)$.			
		Then the operator $ \Psi_\sigma $ extends to a bounded operator on $ L^p{(\mathbb X)} $ to itself and on $ L^{p'}(\X) $ to itself. By interpolation, the operator $ \Psi_\sigma $ is bounded on $ L^2 (\X)$ to itself.  Therefore, if $ \lambda\rightarrow \sigma(x,\lambda) $ is holomorphic and satisfy the inequalities \eqref{eqn:hypothesis_of_sigma-global}  on a strip, then $ \Psi_{\sigma} $ is bounded on $ L^2(\X) $. But the condition on the symbol $\sigma$ above is not natural  for  the $ L^2 $ boundedness of $ \Psi_\sigma $.  In the next section, we prove the $ L^2 $ boundedness of $ \Psi_\sigma$ in the $ K $-biinvariant setting,  assuming only smoothness condition (without holomorphicity) of the symbol $ \sigma $ on the real line. 
	\end{enumerate}
\end{remark}


\section{\textbf{$ L^2 $ boundedness for  symmetric spaces  in $ K $ -biinvariant cases}}\label{L2 theorems proof}

In this section we shall prove the $L^2$ boundedness of pseudo differential operator associated with a symbol $\sigma$ in the $K$-biinvariant setting. We recall that 
 $ \X=G/K $ is a rank one symmetric space with dim $ \X =d$. 
\begin{proof}[\textbf{Proof of the Theorem \ref{thm_for_L2_cap_Lp_to_Lp_K_biinv}}]: Let $ \eta:[0, \infty)\rightarrow [0, 1]$ be a smooth cut off function supported in $[1, \infty)$ and equal to $1$ on $[2, \infty)$.  Using  Proposition \ref{propn:global_expansion_of_phi_lambda}, (\ref{prop:abelslice}) and evenness of  $ \lambda\mapsto \sigma(\cdot, \lambda) $ we  write,
\begin{align*}
	\eta(t) \Psi_\sigma f(a_t) e^{\rho t }& =2\int_{\R} \eta(t) \frac{\sigma(t,\lambda)}{c(-\lambda)} e^{i\lambda t}  \, \what {f} (\lambda) d\lambda + 2\int_{\R} \eta(t) \frac{\sigma(t,\lambda) a(t,\lambda)}{c(-\lambda)} e^{i\lambda t} \, \what {f} (\lambda) d\lambda \\
	& =2\int_{\R} \eta(t) \frac{\sigma(t,\lambda)}{c(-\lambda)} e^{i\lambda t}  \, \mathcal F \mathcal Af (\lambda) d\lambda + 2\int_{\R} \eta(t) \frac{\sigma(t,\lambda) a(t,\lambda)}{c(-\lambda)} e^{i\lambda t} \, \mathcal F \mathcal Af (\lambda) d\lambda \\
	&= \tilde{\Psi}_{a_1} (\mathcal{A} f)(t) + \tilde{\Psi}_{a_2}(\mathcal{A}f )(t),
\end{align*}
where \begin{equation}\label{defn:a_1}
	 a_1(t,\lambda) =2 \eta(t) \frac{\sigma(t,\lambda)}{c(-\lambda)} \, \text{ and }\,  a_2(t,\lambda) =2 \eta(t) \frac{\sigma(t,\lambda)a(t,\lambda)}{c(-\lambda)},
\end{equation}
and $\tilde{\Psi}_{a_1}, \tilde{\Psi}_{a_2}$ are the Euclidean pseudo differential operator associated with symbols $a_1, a_2$ respectively.
Now from \eqref{eqn:condn_of_sigma_for_real_rank_one_L2 case},\eqref{eqn:est_for_a(lambda,r)} and (\ref{est: cminusla})  it follows that $ a_1, a_2: \R\times \R\rightarrow \C  $ are smooth functions  and satisfies the following inequalities:
\begin{align}
	 	\left|\frac{\partial^\beta}{{\partial t}^\beta} \, \frac{\partial^\alpha}{{\partial \lambda}^\alpha} a_{i}(t,\lambda) \right| \leq C_{\alpha,\beta} (1+|\lambda|)^{-\alpha}, 
\end{align}
for all $ \lambda \in \R$, $ \alpha ,\beta \in \Z^+ $ and $ i=1,2. $
Using the results of pseudo differential operators in Euclidean spaces  (Theorem \ref{thm:pdo_on_Rn}), it follows that $  \tilde{\Psi}_{a_i} $ extends to a bounded operator from $ L^2(\R)$ to itself.  Therefore for $ f \in C_c^\infty(G//K) $,
\begin{align}\label{eqn:ineq_||T_sigma_f||_2_leq_||Af||_2}
	\left(\int_{2}^{\infty} |\Psi_\sigma (f)(a_t)|^2 e^{2\rho t} dt\right)^{\frac{1}{2}} &\leq \left(\int_{1}^{\infty}|\eta(t) \Psi_\sigma (f)(a_t)|^2 e^{2\rho t} dt\right)^{\frac{1}{2}} 
	\leq  C \|\mathcal{A}f\|_{L^{2}(\R)} .
\end{align}
Now  we choose $ M> 0 ,$ such that  $ |c(\lambda)|^{-2} \geq 1 $ for all $ |\lambda| \geq M $. Then using the Euclidean Plancherel theorem and (\ref{prop:abelslice}), we have  \begin{align*}
	\|\mathcal{A}f\|_{L^{2}(\R)}^2 &=  \int_{\R}|\what f(\lambda)|^2 \, d\lambda  =  2 \int_{0}^M|\what f(\lambda)|^2 \, d\lambda  + 2 \int_{M}^\infty|\what f(\lambda)|^2 \, d\lambda \\
	& \leq C\|f\|_{L^p(\X)}^2 +\int_{M}^\infty|\what f(\lambda)|^2 |c(\lambda)|^{-2} \, d\lambda \\
	& \leq C \left(\|f\|_{L^p(\X)}^2 + \|f\|_{L^2(\X)}^2\right),
\end{align*}
where we used the fact \[   |\what f(\lambda) | \leq \|f\|_{L^p(\X)} \|\phi_{0}\|_{L^{p'}(\X)}, \] for all  $ \lambda \in \R   $  (see \cite[Propostion 2.1]{PRS}).
Hence  we get 
\begin{equation}
\|\mathcal{A}f\|_{L^{2}(\R)} \leq  C \left(\|f\|_{L^p(\X)} + \|f\|_{L^2(\X)}\right).
\end{equation}
Therefore from \eqref{eqn:ineq_||T_sigma_f||_2_leq_||Af||_2}  we have, 
\begin{equation}\label {eqn:est_||T_sigmaf||_2_leq_||f||_p+||f||_2_away_from_0}
	\left(\int_{2}^{\infty} |\Psi_\sigma (f)(a_t)|^2 e^{2\rho t} dt\right)^{\frac{1}{2}} \leq  C \left(\|f\|_{L^p(\X)} + \|f\|_{L^2(\X)}\right).
\end{equation}
Next from (\ref{eqn:defn_of_Tsigmaf_for_K_biinv_case_real_sym_space}), (\ref{eqn:condn_of_sigma_for_real_rank_one_L2 case})  and Lemma \ref{estimate of |c(-lambda)|^-2}  we get
\begin{align}\label{eqn:est_T_sigma(at)_near_t=0}
	|\Psi_\sigma(f)(a_t) | \leq  \|f\|_{L^{p}(\X)}\| \|{\phi_0}\|_{L^{p'}(\X|)} \int_{\R}| \sigma(t,\lambda)| \, |c(\lambda)|^{-2} d\lambda \leq C \|f\|_{L^{p}(\X)}.
\end{align}
  Therefore from \eqref {eqn:est_||T_sigmaf||_2_leq_||f||_p+||f||_2_away_from_0} and \eqref{eqn:est_T_sigma(at)_near_t=0} we get
\begin{equation*}
\|\Psi_\sigma f\|_{L^2(\X)} \leq  C \left(\|f\|_{L^p(\X)} + \|f\|_{L^2(\X)}\right) \text{ for any } 1\leq p<2.
\end{equation*}
\end{proof}

\begin{proof} [\textbf{Proof of Theorem \ref{thm:pdo_real_sym_space_K_biinv_2<p<3}:}]
Let $ \eta:[0, \infty)\rightarrow [0, 1]$ be a smooth cut off function supported in $[1, \infty)$ and equal to $1$ on $[2, \infty)$. 	Suppose $ 2<q<3 $.  Then from \eqref{eqn:defn_of_Tsigmaf_for_K_biinv_case_real_sym_space}, (\ref{prop:abelslice}) and   Proposition \ref{propn:global_expansion_of_phi_lambda} we have 
\begin{equation}\label{eqn: expression-pgreat2}
	\begin{aligned}
	\eta(t) \Psi_\sigma (f)(a_t) e^{\frac{2\rho}{q}t}& = e^{(\frac{2}{q}-1)\rho t}\int_\R \frac{2\eta(t) \sigma(t,\lambda)}{c(-\lambda)} e^{i\lambda t} \mathcal F \mathcal{A}f (\lambda) d\lambda\\
	 &+ e^{(\frac{2}{q}-1)\rho t}\int_\R \frac{2\eta(t) \sigma(t,\lambda) e^{-2t } a(t,\lambda))}{c(-\lambda)} e^{i\lambda t} \mathcal F \mathcal{A}f(\lambda) d\lambda\\
	 & = e^{(\frac{2}{q}-1)\rho t} {\tilde \Psi_{a_1}(\mathcal{A}f)(t)} +e^{(\frac{2}{q}-1)\rho t} {\tilde \Psi_{a_2}(\mathcal{A}f)(t)},
\end{aligned}
\end{equation}
where $a_1, a_2$ are defined in (\ref{defn:a_1}) and $\tilde{\Psi}_{a_1}, \tilde{\Psi}_{a_2}$ are the Euclidean pseudo differential operator associated with symbols $a_1, a_2$ respectively.
Since $ q>2 $, from above it follows that   \begin{eqnarray*}
\left(\int_{1}^{\infty}|\eta(t) \Psi_\sigma (f)(a_t)|^q e^{2\rho t} dt\right)^{\frac{1}{q}} \leq \|  {\tilde \Psi_{\sigma_1}(\mathcal{A}f)}\|_{L^q(\R)}+   \| {\tilde \Psi_{\sigma_2}(\mathcal{A}f)}\|_{L^q(\R)}
\end{eqnarray*}
 Applying the result for pseudo differential operator in Euclidean case (Theorem \ref{thm:pdo_on_Rn} ) on $\tilde{\Psi}_{a_1} $ and $\tilde{\Psi}_{a_2} $ we get $\tilde{\Psi}_{a_i} :L^q(\R) \rightarrow L^{q}(\R) $ are bounded for $ i=1,2 $, and so
\begin{align}\label{eqn:||T_sigma_f||_p_leq_||Af||_p_away_zero}
	 \left(\int_{2}^{\infty} |\Psi_\sigma (f)(a_t)|^q e^{2\rho t} dt\right)^{\frac{1}{q}} &\leq \left(\int_{1}^{\infty}|\eta(t) \Psi_\sigma (f)(a_t)|^q e^{2\rho t} dt\right)^{\frac{1}{q}} 
	 \leq  C \|\mathcal{A}f\|_{L^{q}(\R)} 
\end{align}
Since $2<q<3$, we put  $ r=q $  and $p=q'$ in \eqref{eqn:result_of_Sk_ray_and_RP_sarkar}, to get
\begin{equation}\label{eqn:mapping-abel}
	 \left(  \int_{\R} |\mathcal{A}f(a_t)|^q   dt \right)^{\frac{1}{q}} \leq  C\|f\|_{L^{q'}(\X)}.
\end{equation}
Therefore from  \eqref{eqn:||T_sigma_f||_p_leq_||Af||_p_away_zero} and (\ref{eqn:mapping-abel}) we get,
\begin{equation}\label{ineq estimate of  integration of T sigma p global}
 \left(\int_{2}^{\infty} |\Psi_\sigma (f)(a_t)|^q e^{2\rho t} dt\right)^{\frac{1}{q}} \leq  C \|f\|_{L^{q'}(\X)}.
\end{equation}
Also from \eqref{eqn:defn_of_Tsigmaf_for_K_biinv_case_real_sym_space}, (\ref{eqn:condn_of_sigma_for_real_rank_one_L2 case}) and  (\ref{estimate of |c(-lambda)|^-2}) we have,
\begin{align*}
	 |\Psi_\sigma(f)(a_t) | \leq  \|f\|_{L^{q'}(\X)} \|{\phi_0}\|_{L^q(\X)} \int_{\R}| \sigma(t,\lambda)| \, |c(\lambda)|^{-2} d\lambda \leq C_p \|f\|_{L^{q'}(\X)}.
\end{align*}
This implies
\begin{align}\label{eqn:||T_sigma_f||_p_leq_||Af||_p_near_zero}
	\left(\int_{0}^{2} |\Psi_\sigma (f)(a_t)|^q e^{2\rho t} dt\right)^{\frac{1}{q}} \leq C \|f\|_{L^{q'}(\X)}.
\end{align}
Finally from \eqref{ineq estimate of  integration of T sigma p global} and  \eqref{eqn:||T_sigma_f||_p_leq_||Af||_p_near_zero} we have for $ 2<q<3 $,
\begin{equation}\label{eqn:est_||T_sigmaf||_p_from_Lp'}
	\|\Psi_\sigma (f)\|_{L^q(\X)} \leq C \|f\|_{L^{q'} (\X)}  .
\end{equation}
Now if we take $ 3/2<p<2 $, then $p=q'$ for some $ 2<q<3 $. So from \eqref{eqn:est_||T_sigmaf||_p_from_Lp'}  we have,
\begin{equation}
	 	\|\Psi_\sigma (f)\|_{L^{p'}(\X)} \leq C \|f\|_{L^{p} (\X)}  .
\end{equation}
Also  from Theorem \ref{thm:pdo_for_sym_space} we have \begin{equation}
	 	\|\Psi_\sigma (f)\|_{L^{p}(\X)} \leq C \|f\|_{L^{p} (\X)}  .
\end{equation}
Thus  by interpolation it follows that for all $ r$ with $ p\leq r\leq p' $  there is a $ C>0  $ such that 
\begin{equation}\label{est:est_||T-sigmaf||_q}
		\|\Psi_\sigma (f)\|_{L^{r}(\X)} \leq C \|f\|_{L^{p} (\X)}  .
\end{equation}
In particular $r=2$ gives,,
\begin{equation}
	 	\|\Psi_\sigma (f)\|_{L^{2}(\X)} \leq C \|f\|_{L^{p} (\X)}  .
\end{equation}
\end{proof}
From the proof of this theorem the following corollary follows:
\begin{corollary}
Let $\frac{3}{2}<p<2$ and $\sigma$ a symbol satisfies condition of Theorem \ref{thm:pdo_real_sym_space_K_biinv_2<p<3}. Then 
\bes
\|\Psi_\sigma(f)\|_{L^r(\X)}\leq C \|f\|_{L^{p}(\X)},
\ees
for $p\leq r\leq p'$.
\end{corollary}
In the proof of last two theorems,  we had to multiply the smooth cut off function $\eta$ (in first step) because we have the estimate of $a(t, \lambda)$ when $t$ is away from origin (see Proposition \ref{propn:global_expansion_of_phi_lambda}). \\

\noindent{\bf Complex symmetric space:} Let $\X$ be an arbitrary rank complex symmetric space. The elementary spherical functions on complex symmetric space $\X$ are given by the expression 
\begin{equation}\label{eqn:defn_of_phi_lambda(H)_in_complex_case}
	\phi_\lambda(H) =c(\lambda) \frac{\sum_{s\in W } \det (s) e^{ -is\lambda(H)}}{\phi(H)}, \, \text{ for } H\in \mathfrak{a}, \lambda \in \mathfrak{a}^*,
\end{equation}
where $ \phi{(H)} $ is defined by the formula
\begin{equation}\label{eqn:defn_of_phi(H)_complex_case}
	\phi(H) =\sum_{s\in W} \det (s) e^{s\rho(H)},\, \text{ for } H\in \mathfrak{a},
\end{equation}
(see \cite[p. 907 and p. 910] {cowling2001}). 
 It is known that the Haar measure on $ A $ corresponding to the polar decomposition of  $G  $ is given by $ \phi(H)^2 dH $ for $ H\in 
\mathfrak{a} $.

Thus the spherical transform of a suitable $K$-biinvariant function $f$ is given by \begin{equation}\label{eqn:defn_of_whatf_in_complex_lie_gorup}
	\what f(\lambda )= \int_{\mathfrak{a}} f(H) \phi_\lambda(H) \phi(H)^2 dH, \,{\text{ for }}\lambda \in \mathfrak{a}^*.
\end{equation}  
We define the $K$-biinvariant pseudo differential operator $\Psi_\sigma$ associated to $ \sigma $ by,
\begin{align}\label{defn:Tsigmacomplex}
	\Psi_\sigma (f)(H) &= \int_{\mathfrak{a}^*} \what f(\lambda) \sigma(H,\lambda) \phi_\lambda(H) |c(\lambda)|^{-2} d\lambda.
	\end{align}

Now we shall prove Theorem \ref{thm:pdo_complex-case}:

\begin{proof}[\textbf{Proof of the Theorem \ref{thm:pdo_complex-case}:}]
	From \eqref{eqn:defn_of_whatf_in_complex_lie_gorup} and \eqref{eqn:defn_of_phi_lambda(H)_in_complex_case} we have
	\begin{align*}\label{eqn:defn_of_whatf_in_complex_lie_gorup}
		\what f(\lambda )&=c(\lambda)\sum_{s\in W} \det(s)\int_{\mathfrak{a}} f(H)  \phi(H)  e^{-is\lambda(H)}   dH\\
		&=  c(\lambda)\sum_{s\in W} \det(s) \int_{\mathfrak{a}} f(H)  \phi(H) e^{-is\lambda(H)}dH\\
		&=  c(\lambda)\sum_{s\in W} \det(s) \mathcal F g(s \lambda), \, \text{ for } \lambda \in \mathfrak{a}^*,
	\end{align*} 
	where the function $ g  $  on $ \mathfrak{a} $ is defined as \begin{equation*}
		g(H) =f(H) \phi(H) , \,  \text{ for }H\in \mathfrak{a},
	\end{equation*}
	and $ \mathcal F g $ denotes the Euclidean Fourier transform of $ g $.  Since $ \phi $ satisfies
	 \bes \phi(sH) = \det (s) \phi(H),
	 \ees for $ H\in \mathfrak{a}, s\in W$, we have \begin{equation}
		g(sH) =\det(s)g(H),\, \text{ for } H\in \mathfrak{a}, s\in W.
	\end{equation}
Therefore,   \begin{equation}\label{eqn:antisymmetrytildeg}
		\mathcal F g(s\lambda) =\det (s) \,\mathcal Fg(\lambda), \, \text{ for } \lambda \in \mathfrak{a}^*, s\in W.
	\end{equation}
	Consequently the spherical transform of $ f $ can be written as
	\begin{equation}
		\what f(\lambda)=c(\lambda) |W| \, \mathcal F g(\lambda), \, \text{ for } \lambda \in \mathfrak{a}^*.
	\end{equation}
	Now from the definition (\ref{defn:Tsigmacomplex}) of $ \Psi_\sigma $ we get,
	\begin{align*}
		\Psi_\sigma (f)(H) &= \int_{\mathfrak{a}^*} \what f(\lambda) \,\sigma(H,\lambda) \phi_\lambda(H) |c(\lambda)|^{-2} d\lambda\\
&= \int_{\mathfrak{a}^*} \what f(\lambda)\, \sigma(H,\lambda) \phi_{-\lambda}(H) |c(\lambda)|^{-2} d\lambda\\		
		& =|W| \int_{\mathfrak{a}^*}  \mathcal F g(\lambda) c(\lambda)  \sigma(H,\lambda) c(-\lambda)  \frac{\sum_{s\in W } \det (s) e^{ is\lambda(H)}}{\phi(H)} |c(\lambda)|^{-2} d\lambda\\
		& =\frac{|W|}{\phi(H)} \sum_{s\in W } \det (s) \int_{\mathfrak{a}^*}  \mathcal F g(\lambda)  \sigma(H,\lambda)  e^{ is\lambda(H)}  d\lambda.
	\end{align*}
Now by the change of variable $\lambda\mapsto s^{-1}\lambda$ and (\ref{eqn:antisymmetrytildeg}) we get
\bes
\phi(H)\Psi_\sigma (f)(H)=|W|^2\int_{\mathfrak{a}^*}\mathcal F g(\lambda)   \sigma(H,\lambda)  e^{ i\lambda(H)}  d\lambda.
\ees 
That is,  that \begin{equation}
		\Psi_\sigma (f)( H) \, \phi(H) =|W|^2\widetilde \Psi_{\sigma}(g)(H),
	\end{equation}
where $\widetilde{T}_\sigma$ is the Euclidean pseudo differential operator associated with the symbol $\sigma(H, \lambda)$.
Using Calder\'on and Vaillancourt theorem \cite[Theorem 5.1]{Tosio76}, we have $ \widetilde{\Psi}_{\sigma}: L^2\rightarrow L^2 $ bounded.  Therefore we have
	\begin{equation}\nonumber
		\int_{\mathfrak{a}}|\Psi_\sigma(f)(H)|^2 \phi(H)^2dH = |W|^4\int_{\mathfrak{a}}| \widetilde{\Psi}_{\sigma}(g)(H)| ^2 dH \leq C \int_{\mathfrak{a}} |g(H)|^2 dH = \int_{\mathfrak{a}} |f(H)|^2 \phi(H)^2 dH.
	\end{equation} 
	This shows that 
	\bes
	\|\Psi_\sigma(f)\|_{L^2(G//K)} \leq C\, \|f\|_{L^2(G//K)}.
	\ees
	This completes the proof.
\end{proof}

\section{appendix}\label{appendix}

\subsection{Coifman-Weiss transference method} (\cite{coifman 73}, \cite{Coifman 77}):
		Suppose $ {G} $ is a locally compact group satisfying the following property: Given a compact subset  $ B $ of $ G $ and $ \epsilon >0  $, there exists an open neighborhood $ V $ of identity having finite measure such that \begin{equation*}
			 \frac{\mu\left( VB^{-1}\right)}{\mu (V)} \leq 1+\epsilon,
		\end{equation*}
where $ \mu $ is a left Haar measure. Any locally compact abelian group satisfies this property. Let $ \mathcal{R} $ be a representation of $ G $ acting on functions on a $ \sigma $ finite measure space $ \mathcal{X} $ satisfying, for some $ p\in [1,\infty] $ 
\begin{equation}
	 \int_{\mathcal{X} } |(\mathcal{R}_u f) (x)|^p d\mu(x) \leq C   \int_{\mathcal{X} } |f (x)|^p d\mu(x),
\end{equation}
where $ C $ is independent of $ f \in L^p(\mathcal{X}) $ and $ u\in G $. Consider the  transformation of  the following form
\begin{equation}\label{eqn:reprn_of_T_using_K}
	 \mathcal{T} f(x) = \int_{G}   \mathfrak{K} (x,R_u x,u ) f(\mathcal{R}_u x) d\mu(u),
\end{equation}
			where $ \mathfrak{K} (x,y,u) $ is a measurable function  on $ \mathcal{X} \times \mathcal{X} \times G $, which is zero if $ u $ does not belong to a compact set $ B\subset G $. Moreover for each $ x \in \mathcal{X}   $, the kernel 
			\begin{equation*}
				 \mathfrak{K}_x (u,v) = \mathfrak{K}(\mathcal{R}_v x, \mathcal{R}_{u^{-1}}\mathcal{R}_v x, u)=\mathfrak{K}(\mathcal{R}_v x, \mathcal{R}_{u^{-1}v}x, u)
			\end{equation*}
		satisfies
		\begin{equation}\label{eqn:req_est_for_pi}
			 \left( \int_G \left|  \int_G  \mathfrak{K}_x(v,u) g (u^{-1}v ) d\mu(u) \right|^p d\mu(v)  \right)^{\frac{1}{p}} \leq A \left(\int_G  \left| g(u)\right|^p d\mu(u)\right)^{\frac{1}{p}},	
		 \end{equation}
	 where $ A $ is independent of $ x \in \mathcal{X} $ and $ g\in L^p(G) $. Then  $ \mathcal{T} $  is a bounded operator on $ L^p(\mathcal{X}) $ with norm not exceeding $ A $ (\cite[(2.7)]{Coifman 77}):
	 \begin{equation}\label{eqn:Lp_est_for_Pi}
	  \left(\int_{\mathcal{X}}  \left| \mathcal{T} f(x)\right|^p dx\right)^{\frac{1}{p}}	  \leq A \left(\int_{\mathcal{X} } \left| f(x)\right|^p dx\right)^{\frac{1}{p}}.
	 \end{equation}
 \section{Open Problem}
 	It will be interesting  to get an exact analogue of $L^2$ boundedness of pseudo differential operators on rank one symmetric spaces and also $L^p$ boundedness of pseudo differential operators on (arbitrary rank) symmetric spaces of noncompact type. This will be taken as our future project.

\noindent{\bf Acknowledgement:} We are thankful to Mithun Bhowmik, Rahul Garg and Pratyoosh Kumar for many helpful discussions on the problem. 

\end{document}